\documentclass{article}%
\usepackage{amsmath}
\usepackage{amsfonts}
\usepackage{amssymb}
\usepackage{graphicx}%
\providecommand{\U}[1]{\protect\rule{.1in}{.1in}}
\newtheorem{theorem}{Theorem}

\newtheorem{definition}[theorem]{Definition}
\newtheorem{example}[theorem]{Example}

\newtheorem{lemma}[theorem]{Lemma}

\newtheorem{remark}[theorem]{Remark}

\newenvironment{proof}[1][Proof]{\noindent \textbf{#1.} }{\  \rule{0.5em}{0.5em}}
\begin{document}

\title{{\Large Dynamic Programming Principle for Stochastic Recursive Optimal Control
Problem under }$G${\Large -framework}}
\author{Mingshang Hu \thanks{Qilu Institute of Finance, Shandong University, Jinan,
Shandong 250100, PR China. humingshang@sdu.edu.cn. Research supported by NSF
(No. 11201262, 11101242 and 11301068) and Shandong Province (No. BS2013SF020)}
\and Shaolin Ji\thanks{Qilu Institute of Finance, Shandong University, Jinan,
Shandong 250100, PR China. jsl@sdu.edu.cn (Corresponding author). Research
supported by NSF (No. 11171187, 11222110 and 11221061), Shandong Province
(No.JQ201202), Programme of Introducing Talents of Discipline to Universities
of China (No.B12023) and Program for New Century Excellent Talents in
University of China. Hu and Ji's research was partially supported by NSF (No.
10921101) and by the 111 Project (No. B12023) }}
\date{}
\maketitle

\textbf{Abstract}. In this paper, we study a stochastic recursive optimal
control problem in which the cost functional is described by the solution of a
backward stochastic differential equation driven by $G$-Brownian motion. Under
standard assumptions, we establish the dynamic programming principle and the
related fully nonlinear HJB equation in the framework of $G$-expectation.
Finally, we show that the value function is the viscosity solution of the
obtained HJB equation.



\addcontentsline{toc}{section}{\hspace*{1.8em}Abstract}

\section{Introduction}

It is well known that Duffie and Epstein \cite{DE} introduced a stochastic
differential recursive utility which corresponds to the solution of a
particular backward stochastic differential equation (BSDE). Thus the BSDE
point of view gives a simple formulation of recursive utilities (see
\cite{EPQ}). Since then, the classical stochastic optimal control problem is
generalized to a so called "stochastic recursive optimal control problem" in
which the cost functional is defined by the solution of BSDE. The stochastic
maximum principle and dynamic programming principle for this problem were
first established in Peng \cite{peng-1993} and \cite{peng-dpp} respectively.

Recently Hu et. al studied a new kind of BSDE which is driven by $G$-Brownian
motion in \cite{HJPS} and \cite{HJPS1}:%
\begin{align}
Y_{t}  &  =\xi+\int_{t}^{T}f(s,Y_{s},Z_{s})ds+\int_{t}^{T}g(s,Y_{s}%
,Z_{s})d\langle B\rangle_{s}\label{state-intro}\\
&  -\int_{t}^{T}Z_{s}dB_{s}-(K_{T}-K_{t}).\nonumber
\end{align}
They proved that there exists a unique triple of processes $(Y,Z,K)$ which
solves (\ref{state-intro}) under the standard Lipschitz conditions. This new
kind of BSDE is based on the $G$-expectation theory which is introduced by
Peng (see \cite{P07a}, \cite{P10} and the references therein). This
$G$-expectation framework ($G$-framework for short) does not require the
probability space and is convenient to study financial problems involving
volatility uncertainty. Let us mention that there are other recent advances in
this direction. Denis, Martini \cite{DenisMartini2006} and Denis, Hu, Peng
\cite{DHP11} developed quasi-sure stochastic analysis. Soner et al.
\cite{STZ11} have obtained a existence and uniqueness theorem for a new type
of fully nonlinear BSDE, called 2BSDE.

An important property of the solution $Y$ of (\ref{state-intro}) is that it
can be represented as the "supremum of expectations" over a set of
nondominated probability measures. For example, the solution $Y$ of
(\ref{state-intro}) at time $0$ can be written as
\begin{align}
Y_{0}  &  =\mathbb{\hat{E}}[\xi+\int_{0}^{T}f(s,Y_{s},Z_{s})ds+\int_{0}%
^{T}g(s,Y_{s},Z_{s})d\langle B\rangle_{s}]\label{state-intro-1}\\
&  =\sup_{P\in\mathcal{P}}E_{P}[\xi+\int_{0}^{T}f(s,Y_{s},Z_{s})ds+\int%
_{0}^{T}g(s,Y_{s},Z_{s})d\langle B\rangle_{s}]\nonumber
\end{align}
where $\mathcal{P}$ is a family of weakly compact nondominant probability
measures. Then, (\ref{state-intro}) can be used to define recursive utility
under volatility uncertainty. It is worth to point out that the recursive
utility under mean uncertainty was developed in Chen and Epstein \cite{CE}.
Epstein and Ji \cite{EJ-1,EJ-2} introduced a particular recursive utility
under both mean and volatility uncertainty.

Motivated by the recursive utility optimization under volatility uncertainty,
we explore a stochastic recursive optimal control problem in which the cost
functional is defined by the solution of the above new type of BSDE. In more
details, the state equation is governed by the following controlled SDE driven
by $G$-Brownian motion
\begin{align*}
dX_{s}^{t,x,u}  &  =b(s,X_{s}^{t,x,u},u_{s})ds+h_{ij}(s,X_{s}^{t,x,u}%
,u_{s})d\langle B^{i},B^{j}\rangle_{s}+\sigma(s,X_{s}^{t,x,u},u_{s})dB_{s},\\
X_{t}^{t,x,u}  &  =x.
\end{align*}
The cost functional is introduced by the solution $Y_{t}^{t,x,u}$ of the
following BSDE driven by $G$-Brownian motion at time $t$:%
\[%
\begin{array}
[c]{rl}%
-dY_{s}^{t,x,u}= & f(s,X_{s}^{t,x,u},Y_{s}^{t,x,u},Z_{s}^{t,x,u}%
,u_{s})ds+g_{ij}(s,X_{s}^{t,x,u},Y_{s}^{t,x,u},Z_{s}^{t,x,u},u_{s})d\langle
B^{i},B^{j}\rangle_{s}\\
& -Z_{s}^{t,x,u}dB_{s}-dK_{s}^{t,x,u},\\
Y_{T}^{t,x,u}= & \Phi(X_{T}^{t,x,u}),\text{ \ \ }s\in\lbrack t,T]\text{.}%
\end{array}
\]
We define the value function of our stochastic recursive optimal control
problem as follows:
\[
V(t,x)=\underset{u(\cdot)\in\mathcal{U}[t,T]}{\text{ess}\inf}Y_{t}^{t,x,u},
\]
where the control set is in the $G$-framework. In view of (\ref{state-intro-1}%
), we essentially have to solve a "inf sup problem". Such problem is known as
the robust optimal control problem, i.e., we consider the worst scenario by
maximizing over a set of probability measures and then we minimize the cost
functional. For recent development of robust utility maximization under
volatility uncertainty, we refer the interested readers to \cite{TTU},
\cite{MPZ} and \cite{DK}. Tevzadze, Toronjadze, Uzunashvili \cite{TTU} studied
robust exponential and power utilities. Matoussi, Possamai, Zhou \cite{MPZ}
related robust utility maximization problem to a particular 2BSDE with
quadratic growth. In \cite{DK}, Denis and Kervarec established a duality
theory for this problem in nondominated models.

The objective of our paper is to establish the dynamic programming principle
(DPP) for this stochastic recursive optimal control problem and investigate
the value function in $G$-framework.

It is well known that DPP and related HJB equations is a powerful approach to
solving optimal control problems (see \cite{Fleming W.H}, \cite{J.Yong} and
\cite{peng-dpp}). For the classical stochastic recursive optimal control
problem, Peng \cite{peng-dpp} obtained the Hamilton--Jacobi--Bellman equation
and proved that the value function is its viscosity solution. In
\cite{peng-dpp-1}, Peng generalized his results and originally introduced the
notion of stochastic backward semigroups which allows him to prove DPP in a
very straightforward way. This backward semigroup approach is also introduced
in the theory of stochastic differential games by Buckdahn and Li in
\cite{BL}. Note that Buckdahn et al. \cite{BLRT} obtained an existence result
of the stochastic recursive optimal control problem.

In this paper, we adopt the backward semigroup approach to build the DPP in
our context. At first, we need to define the essential infimum of a family of
random variables in the \textquotedblleft quasi-surely\textquotedblright%
\ sense (q.s. for short). Compared with classical case in \cite{peng-dpp-1},
this kind of essential infimum may not exist in our case (the q.s.\ case). We
define the essential infimum and prove its existence in this paper. Under a
family of non-dominated probability measures, it is far from being trivial to
prove that the value function $V$ is wellposed and deterministic. Due to a new
result in \cite{HWZ}, we construct the approximation of an element of the
admissible control set which is the key step to prove that $\underset{u(\cdot
)\in\mathcal{U}[t,T]}{\text{ess}\inf}Y_{t}^{t,x,u}$ is a deterministic
function. At last, we adopt an ``implied partition" approach to prove DPP (see
Lemma \ref{le-new-dpp1}) which is completely new in the literature.

We states that $V$ is deterministic continuous viscosity solution of the
following fully nonlinear HJB equation%
\begin{align*}
&  \partial_{t}V(t,x)+\underset{u\in U}{\inf}H(t,x,V,\partial_{x}%
V,\partial_{xx}^{2}V,u)=0,\\
&  V(T,x)=\Phi(x),\quad\ \ x\in\mathbb{R}^{n},
\end{align*}
where%
\[%
\begin{array}
[c]{cl}%
H(t,x,v,p,A,u)= & G(F(t,x,v,p,A,u))+\langle p,b(t,x,u)\rangle+f(t,x,v,\sigma
^{T}(t,x,u)p,u),\\
F_{ij}(t,x,v,p,A,u)= & (\sigma^{T}(t,x,u)A\sigma(t,x,u))_{ij}+2\langle
p,h_{ij}(t,x,u)\rangle\\
& +2g_{ij}(t,x,v,\sigma^{T}(t,x,u)p,u),
\end{array}
\]
$(t,x,v,p,A,u)\in\lbrack0,T]\times\mathbb{R}^{n}\times\mathbb{R}%
\times\mathbb{R}^{n}\times\mathbb{S}_{n}\times U$. The main difficulty to
prove this statement lies in the appearance of two decreasing $G$-martingale
terms. Applying a property of decreasing $G$-martingale proved in Lemma
\ref{le-hjb-4}, we overcome this difficulty (see Lemma \ref{Lem-auxilary-3})
and obtain the result.

In conclusion, since there is no reference probability measure under the
$G$-framework, our results generalize the results in Peng \cite{peng-dpp} and
\cite{peng-dpp-1} which was only considered in the Wiener space (corresponding
to $G$ is linear in our paper). Compared with our earlier article \cite{HJY},
the problem in \cite{HJY} is essentially a "sup sup problem" which is easier
to deal with. And the techniques developped in this paper can also used to
solve the problem in \cite{HJY}. Note that $G$ has the representation
(\ref{G-representation}) which leads to that the above HJB equation can also
be understood as a kind of Bellman-Issac equation. Then, it is meaningful to
show the difference between our paper and some related references (see
\cite{BL} and \cite{PZ}) in game theory. Needless to say, the game problem is
more complicated than the robust control problem since it needs to study the
value of game. Buckdahn, Li \cite{BL} employed strategies and Pham, Zhang
\cite{PZ} formulated their game problem in a weak framework. In constract, we
use controls and our formulation is a "strong" framework under the
$G$-framework. Different from \cite{PZ}, as revealed in \cite{HWZ}, our
admissible control set has quasi-continuous property and in particular, it
does not change with time. It is worth mentioning that, in our context, the
coefficents of the state equation include the state variable $X$.

The paper is organized as follows. In section 2, we present some fundamental
results on $G$-expectation theory. We formulate our stochastic recursive
optimal control problem in section 3. We prove the properties of the value
function in section 4 and establish the dynamic programming principle in
section 5. In section 6, we first derive the fully nonlinear HJB equation and
prove that the value function is the viscosity solution of the obtained HJB equation.

\section{Preliminaries}

We review some basic notions and results of $G$-expectation and the related
spaces of random variables. The readers may refer to \cite{P07a}, \cite{P07b},
\cite{P08a}, \cite{P10} for more details.

Let $\Omega_{T}=C_{0}([0,T];\mathbb{R}^{d})$, the space of $R^{d}$-valued
continuous functions on $[0,T]$ with $\omega_{0}=0$, and $B_{t}(\omega
)=\omega_{t}$ be the canonical process. Set
\[
L_{ip}(\Omega_{T}):=\{ \varphi(B_{t_{1}},...,B_{t_{n}}):n\geq1,t_{1}%
,...,t_{n}\in\lbrack0,T],\varphi\in C_{b.Lip}(\mathbb{R}^{d\times n})\},
\]
where $C_{b.Lip}(\mathbb{R}^{d\times n})$ denotes the set of bounded Lipschitz
functions on $\mathbb{R}^{d\times n}$.

We denote the $G$-expectation space by $(\Omega_{T},L_{ip}(\Omega
_{T}),\mathbb{\hat{E}})$. The function $G:\mathbb{S}_{d}\rightarrow\mathbb{R}$
is defined by
\begin{equation}
G(A):=\frac{1}{2}\mathbb{\hat{E}}[\langle AB_{1},B_{1}\rangle]. \label{eq-G}%
\end{equation}
where $\mathbb{S}_{d}$ denotes the collection of $d\times d$ symmetric
matrices. Note that there exists a bounded and closed subset $\Gamma
\subset\mathbb{R}^{d\times d}$ such that%
\begin{equation}
G(A)=\frac{1}{2}\underset{Q\in\Gamma}{\sup}tr[AQQ^{T}].
\label{G-representation}%
\end{equation}
In this paper, we only consider non-degenerate $G$-normal distribution, i.e.,
there exists some $\underline{\sigma}^{2}>0$ such that $G(A)-G(B)\geq\frac
{1}{2}\underline{\sigma}^{2}\mathrm{tr}[A-B]$ for any $A\geq B$.

We denote by $L_{G}^{p}(\Omega_{T})$ the completion of $L_{ip}(\Omega_{T})$
under the norm $\Vert X\Vert_{p,G}=(\mathbb{\hat{E}}[|X|^{p}])^{1/p}$ for
$p\geq1$. For each$\ t\geq0$, the conditional $G$-expectation $\mathbb{\hat
{E}}_{t}[\cdot]$ can be extended continuously to $L_{G}^{1}(\Omega_{T})$ under
the norm $\Vert\cdot\Vert_{1,G}$.

\begin{definition}
\label{def2.6} Let $M_{G}^{0}(0,T)$ be the collection of processes in the
following form: for a given partition $\{t_{0},\cdot\cdot\cdot,t_{N}\}=\pi
_{T}$ of $[0,T]$,
\[
\eta_{t}(\omega)=\sum_{j=0}^{N-1}\xi_{j}(\omega)I_{[t_{j},t_{j+1})}(t),
\]
where $\xi_{i}\in L_{ip}(\Omega_{t_{i}})$, $i=0,1,2,\cdot\cdot\cdot,N-1$.
\end{definition}

We denote by $M_{G}^{p}(0,T)$ the completion of $M_{G}^{0}(0,T)$ under the
norm $\Vert\eta\Vert_{M_{G}^{p}}=\{ \mathbb{\hat{E}}[\int_{0}^{T}|\eta
_{s}|^{p}ds]\}^{1/p}$ for $p\geq1$.

\begin{theorem}
\label{the2.7 copy(1)} (\cite{DHP11,HP09}) There exists a family of weakly
compact probability measures $\mathcal{P}$ on $(\Omega,\mathcal{B}(\Omega))$
such that
\[
\mathbb{\hat{E}}[\xi]=\sup_{P\in\mathcal{P}}E_{P}[\xi]\ \ \text{for
\ all}\ \xi\in L_{G}^{1}(\Omega).
\]
$\mathcal{P}$ is called a set that represents $\mathbb{\hat{E}}$.
\end{theorem}

For this $\mathcal{P}$, we define capacity%
\[
c(A):=\sup_{P\in\mathcal{P}}P(A),\ A\in\mathcal{B}(\Omega_{T}).
\]
A set $A\subset\Omega_{T}$ is polar if $c(A)=0$. A property holds
\textquotedblleft quasi-surely\textquotedblright\ (q.s. for short) if it holds
outside a polar set. In the following, we do not distinguish two random
variables $X$ and $Y$ if $X=Y$ q.s.. We set%
\[
\mathbb{L}^{p}(\Omega_{T}):=\{X\in\mathcal{B}(\Omega_{T}):\sup_{P\in
\mathcal{P}}E_{P}[|X|^{p}]<\infty\} \ \text{for}\ p\geq1.
\]
It is important to note that $L_{G}^{p}(\Omega_{T})\subset\mathbb{L}%
^{p}(\Omega_{T})$. We extend $G$-expectation $\mathbb{\hat{E}}$ to
$\mathbb{L}^{p}(\Omega_{T})$ and still denote it by $\mathbb{\hat{E}}$, for
each $X\in$ $\mathbb{L}^{1}(\Omega_{T})$, we set%
\[
\mathbb{\hat{E}}[X]=\sup_{P\in\mathcal{P}}E_{P}[X].
\]
For $p\geq1$, $\mathbb{L}^{p}(\Omega_{T})$ is a Banach space under the norm
$(\mathbb{\hat{E}}[|\cdot|^{p}])^{1/p}$.

Furthermore, we extend the definition of conditional $G$-expectation. For each
fixed $t\in\lbrack0,T],$ let $(A_{i})_{i=1}^{n}$ be a partition of
$\mathcal{B}(\Omega_{t})$, and set
\[
\xi=\sum_{i=1}^{n}\eta_{i}I_{A_{i}},
\]
where $\eta_{i}\in L_{G}^{1}(\Omega_{T})$, $i=1,\cdots,n$. We define the
corresponding generalized conditional $G$-expectation, still denoted by
$\mathbb{\hat{E}}_{s}[\cdot]$, by setting%
\[
\mathbb{\hat{E}}_{s}[\sum_{i=1}^{n}\eta_{i}I_{A_{i}}]:=\sum_{i=1}%
^{n}\mathbb{\hat{E}}_{s}[\eta_{i}]I_{A_{i}}\ \text{\ for}\ s\in\lbrack t,T].
\]
Then, many properties of the conditional $G$-expectation still hold (refer to
Proposition 2.5 in \cite{HJPS1}).

\section{Problem}

\subsection{State equations}

We first give the definition of admissible controls.

\begin{definition}
For each $t\in\lbrack0,T],$ $u$ is said to be an admissible control on
$[t,T]$, if it satisfies the following conditions:

(i) $u:[t,T]\times\Omega\rightarrow U$ where $U$ is a given compact set of
$\mathbb{R}^{m}$;

(ii) $u\in M_{G}^{2}(t,T;\mathbb{R}^{m})$.
\end{definition}

The set of admissible controls on $[t,T]$ is denoted by $\mathcal{U}[t,T]$. In
the rest of this paper, we use Einstein summation convention.

Let $t\in\lbrack0,T]$, $\xi\in\cup_{\varepsilon>0}L_{G}^{2+\varepsilon}%
(\Omega_{t};\mathbb{R}^{n})$ and $u\in\mathcal{U}[t,T]$. Consider the
following forward and backward SDEs driven by $G$-Brownian motion:%

\begin{align}
dX_{s}^{t,\xi,u}  &  =b(s,X_{s}^{t,\xi,u},u_{s})ds+h_{ij}(s,X_{s}^{t,\xi
,u},u_{s})d\langle B^{i},B^{j}\rangle_{s}+\sigma(s,X_{s}^{t,\xi,u}%
,u_{s})dB_{s},\label{state-1}\\
X_{t}^{t,\xi,u}  &  =\xi,\nonumber
\end{align}
and%
\begin{align}
-dY_{s}^{t,\xi,u}  &  =f(s,X_{s}^{t,\xi,u},Y_{s}^{t,\xi,u},Z_{s}^{t,\xi
,u},u_{s})ds+g_{ij}(s,X_{s}^{t,\xi,u},Y_{s}^{t,\xi,u},Z_{s}^{t,\xi,u}%
,u_{s})d\langle B^{i},B^{j}\rangle_{s}\label{state-2}\\
&  -Z_{s}^{t,\xi,u}dB_{s}-dK_{s}^{t,\xi,u},\nonumber\\
Y_{T}^{t,\xi,u}  &  =\Phi(X_{T}^{t,\xi,u}),\text{ \ \ \ }s\in\lbrack
t,T]\text{.}\nonumber
\end{align}

Set
\[
S_{G}^{0}(0,T):=\{h(t,B_{t_{1}\wedge t},\cdot\cdot\cdot,B_{t_{n}\wedge
t}):t_{1},\ldots,t_{n}\in\lbrack0,T],h\in C_{b,Lip}(\mathbb{R}^{n+1})\}.
\]
For $p\geq1$ and $\eta\in S_{G}^{0}(0,T)$, let $\Vert\eta\Vert_{S_{G}^{p}}=\{
\mathbb{\hat{E}}[\underset{t\in\lbrack0,T]}{\sup}|\eta_{t}|^{p}]\}^{\frac
{1}{p}}$. Denote by $S_{G}^{p}(0,T)$ the completion of $S_{G}^{0}(0,T)$ under
the norm $\Vert\cdot\Vert_{S_{G}^{p}}$.

For given $t,$ $u$ and $\xi$, $(X^{t,\xi,u})$ and $(Y^{t,\xi,u},Z^{t,\xi
,u},K^{t,\xi,u})$ are called solutions of the above forward and backward SDEs
respectively if

(i) $(X^{t,\xi,u})\in M_{G}^{2}(t,T;\mathbb{R}^{n})$;

(ii) $(Y^{t,\xi,u},Z^{t,\xi,u})\in S_{G}^{2}(0,T)\times M_{G}^{2}%
(0,T;\mathbb{R}^{d})$;

(iii) $K^{t,\xi,u}$ is a decreasing $G$-martingale$\ $with $K_{t}^{t,\xi
,u}=0,$ $K_{T}^{t,\xi,u}\in L_{G}^{2}(\Omega_{T})$;

(iv) (\ref{state-1}) and (\ref{state-2}) are satisfied respectively.

We assume that $b,h_{ij}:[0,T]\times\mathbb{R}^{n}\times U\rightarrow
\mathbb{R}^{n}$, $\sigma:[0,T]\times\mathbb{R}^{n}\times U\rightarrow
\mathbb{R}^{n\times d}$, $\Phi:\mathbb{R}^{n}\rightarrow\mathbb{R}$,
$f,g_{ij}:[0,T]\times\mathbb{R}^{n}\times\mathbb{R}\times\mathbb{R}^{d}\times
U\rightarrow\mathbb{R}$ are deterministic functions and satisfy the following conditions:

\begin{description}
\item[(A1)] There exists a constant $C>0$ such that $\forall(s,x,y,z,u),$
$(s,x^{\prime},y^{\prime},z^{\prime},v)$ $\in\lbrack0,T]\times\mathbb{R}%
^{n}\times\mathbb{R}\times\mathbb{R}^{d}\times U,$%
\begin{align*}
&  |b(s,x,u)-b(s,x^{\prime},v)|+|h_{ij}(s,x,u)-h_{ij}(s,x^{\prime}%
,v)|+|\sigma(s,x,u)-\sigma(s,x^{\prime},v)|\\
&  \leq C(|x-x^{\prime}|+|u-v|),
\end{align*}%
\[
|\Phi(x)-\Phi(x^{\prime})|\leq C|x-x^{\prime}|,
\]%
\begin{align*}
&  |f(s,x,y,z,u)-f(s,x^{\prime},y^{\prime},z^{\prime},v)|+|g_{ij}%
(s,x,y,z,u)-g_{ij}(s,x^{\prime},y^{\prime},z^{\prime},v)|\\
&  \leq C(|x-x^{\prime}|+|y-y^{\prime}|+|z-z^{\prime}|+|u-v|);
\end{align*}

\item[(A2)] $b,h_{ij},\sigma,f,g_{ij}$ are continuous in $s$.
\end{description}

Then, we have the following theorems.

\begin{theorem}
(\cite{P10}) Let Assumptions (A1) and (A2) hold. Then there exists a unique
adapted solution $X$ for equation (\ref{state-1}).
\end{theorem}

\begin{theorem}
\label{SDE-est} (\cite{P10})Let $\xi,\xi^{\prime}\in L_{G}^{p}(\Omega
_{t};\mathbb{R}^{n})$ with $p\geq2$ and $u,v\in\mathcal{U}[t,T]$. Then we
have, for each $\delta\in\lbrack0,T-t]$,%
\[
\mathbb{\hat{E}}_{t}[|X_{t+\delta}^{t,\xi,u}-X_{t+\delta}^{t,\xi^{\prime}%
,v}|^{2}]\leq\bar{C}(|\xi-\xi^{\prime}|^{2}+\mathbb{\hat{E}}_{t}[\int%
_{t}^{t+\delta}|u_{s}-v_{s}|^{2}ds]),
\]%
\[
\mathbb{\hat{E}}_{t}[|X_{t+\delta}^{t,\xi,u}|^{p}]\leq\bar{C}(1+|\xi|^{p}),
\]%
\[
\mathbb{\hat{E}}_{t}[\sup_{s\in\lbrack t,t+\delta]}|X_{s}^{t,\xi,u}-\xi
|^{p}]\leq\bar{C}(1+|\xi|^{p})\delta^{p/2},
\]
where the constant $\bar{C}$ depends on $C$, $G$, $p$, $n$, $U$ and $T$.
\end{theorem}

\begin{theorem}
(\cite{HJPS1}) Let Assumptions (A1) and (A2) hold. Then there exists a unique
adapted solution $(Y,Z,K)$ for equation (\ref{state-2}).
\end{theorem}

\begin{theorem}
\label{BSDE-est} (\cite{HJPS1})Let $\xi,\xi^{\prime}\in\cup_{\varepsilon
>0}L_{G}^{2+\varepsilon}(\Omega_{t};\mathbb{R}^{n})$ and $u,v\in
\mathcal{U}[t,T]$. Then there exist two positive constants $\bar{C}_{1}$ and
$\bar{C}_{2}$ depending on $C$, $G$ and $T$ such that%
\begin{align*}
|Y_{t}^{t,\xi,u}-Y_{t}^{t,\xi^{\prime},v}|^{2}  &  \leq\bar{C}_{1}%
\mathbb{\hat{E}}_{t}[|\Phi(X_{T}^{t,\xi,u})-\Phi(X_{T}^{t,\xi^{\prime}%
,v})|^{2}+(\int_{t}^{T}\hat{F}_{s}ds)^{2}]\\
&  \leq\bar{C}_{2}\mathbb{\hat{E}}_{t}[|\Phi(X_{T}^{t,\xi,u})-\Phi
(X_{T}^{t,\xi^{\prime},v})|^{2}+\int_{t}^{T}|\hat{F}_{s}|^{2}ds],
\end{align*}
where%
\begin{align*}
\hat{F}_{s}  &  =|f(s,X_{s}^{t,\xi,u},Y_{s}^{t,\xi,u},Z_{s}^{t,\xi,u}%
,u_{s})-f(s,X_{s}^{t,\xi^{\prime},v},Y_{s}^{t,\xi,u},Z_{s}^{t,\xi,u},v_{s})|\\
&  +\sum_{i,j=1}^{d}|g_{ij}(s,X_{s}^{t,\xi,u},Y_{s}^{t,\xi,u},Z_{s}^{t,\xi
,u},u_{s})-g_{ij}(s,X_{s}^{t,\xi^{\prime},v},Y_{s}^{t,\xi,u},Z_{s}^{t,\xi
,u},v_{s})|.
\end{align*}

\end{theorem}

\begin{theorem}
\label{quasi}(\cite{HWZ})Let $b$, $h_{ij}$, $\sigma$ be independent of $u$ and
satisfy (A1) and (A2). Assume further that there exist constants $L>0$,
$0<\lambda<\Lambda$ such that $|b|\leq L$, $|h_{ij}|\leq L$, $\lambda
\leq|\sigma_{i}|\leq\Lambda$ for $i\leq n$, where $\sigma_{i}$ is the $i$-th
row of $\sigma$. Then for each $x$, $a$, $a^{\prime}\in\mathbb{R}^{n}$ with
$a\leq a^{\prime}$, $s\geq t$, we have $I_{\{X_{s}^{t,x}\in\lbrack
a,a^{\prime})\}}\in L_{G}^{2}(\Omega_{s})$. In particular, for each $c$,
$c^{\prime}\in\mathbb{R}^{d\times k}$, $c\leq c^{\prime}$ and $t\leq s_{1}%
\leq\cdots\leq s_{k}$, we have $I_{\{(B_{s_{1}}-B_{t},\ldots,B_{s_{k}}%
-B_{t})\in\lbrack c,c^{\prime})\}}\in L_{G}^{2}(\Omega_{s_{k}})$.
\end{theorem}

\begin{remark}
If there exists a $t_{0}<T$ such that $b$, $h_{ij}$, $\sigma$ are continuous
in $s$ just on $[t_{0},T]$, then the above theorem still holds by the proof in
\cite{HWZ}.
\end{remark}

\subsection{Stochastic optimal control problem}

The state equation of our stochastic optimal control problem is governed by
the above forward SDE (\ref{state-1}) and the objective functional is
introduced by the solution of the BSDE (\ref{state-2}) at time $t$. Let $\xi$
equals a constant $x\in\mathbb{R}^{n}$. When $u$ changes, $Y_{t}^{t,x,u}$ (the
solution $Y^{t,x,u}$ at time $t$) also changes. In order to study the value
function of our stochastic optimal control problem, we need to define the
essential infimum of $\{Y_{t}^{t,x,u}\mid u\in\mathcal{U}[t,T]\}.$

\begin{definition}
\label{esssup}The essential infimum of $\{Y_{t}^{t,x,u}\mid u\in
\mathcal{U}[t,T]\}$, denoted by $\underset{u(\cdot)\in\mathcal{U}%
[t,T]}{\text{ess}\inf}Y_{t}^{t,x,u}$, is a random variable $\zeta\in L_{G}%
^{2}(\Omega_{t})$ satisfying:

(i) $\forall u\in\mathcal{U}[t,T],$ $\zeta\leq Y_{t}^{t,x,u}$ $\ $q.s.;

(ii) if $\eta$ is a random variable satisfying $\eta\leq Y_{t}^{t,x,u}$
$\ $q.s. for any $u\in\mathcal{U}[t,T]$, then $\zeta\geq\eta$ $\ $q.s..
\end{definition}

Similarly, we can define the essential infimum of $\{Y_{t}^{t,\xi,u}\mid
u\in\mathcal{U}[t,T]\},$ where $\xi\in\cup_{\varepsilon>0}L_{G}^{2+\varepsilon
}(\Omega_{t};\mathbb{R}^{n})$.

The following example shows that the essential infimum may be not exist.

\begin{example}
Let $d=1$ and $(B_{t})_{t\geq0}$ be a $1$-dimensional $G$-Brownian motion with
$G(a)=\frac{1}{2}(a^{+}-\frac{1}{3}a^{-})$. We first show that $I_{\{\langle
B\rangle_{1}=\frac{1}{2}\}}$, $I_{\{\langle B\rangle_{1}\geq\frac{1}{2}\}}$
and $I_{\{\langle B\rangle_{1}>\frac{1}{2}\}}\not \in L_{G}^{2}(\Omega_{1})$.

It is easy to verify that $h_{k}(\langle B\rangle_{1})\downarrow I_{\{\langle
B\rangle_{1}=\frac{1}{2}\}}$, where
\[
h_{k}(x)=k(x-\frac{1}{2}+\frac{1}{k})I_{[\frac{1}{2}-\frac{1}{k},\frac{1}{2}%
]}(x)+[1-k(x-\frac{1}{2})]I_{(\frac{1}{2},\frac{1}{2}+\frac{1}{k}]}(x).
\]
If $I_{\{\langle B\rangle_{1}=\frac{1}{2}\}}\in L_{G}^{2}(\Omega_{1})$, then
\[
h_{k}(\langle B\rangle_{1})-I_{\{\langle B\rangle_{1}=\frac{1}{2}\}}\in
L_{G}^{2}(\Omega_{1})\text{ and }h_{k}(\langle B\rangle_{1})-I_{\{\langle
B\rangle_{1}=\frac{1}{2}\}}\downarrow0.
\]
By Corollary 33 in \cite{DHP11}, we have $\mathbb{\hat{E}}[h_{k}(\langle
B\rangle_{1})-I_{\{\langle B\rangle_{1}=\frac{1}{2}\}}]\downarrow0$. On the
other hand,%
\[%
\begin{array}
[c]{rl}
& \mathbb{\hat{E}}[h_{k}(\langle B\rangle_{1})-I_{\{\langle B\rangle_{1}%
=\frac{1}{2}\}}]\\
\geq & \lim_{j\rightarrow\infty}\mathbb{\hat{E}}[h_{k}(\langle B\rangle
_{1})-h_{j}(\langle B\rangle_{1})]\\
= & \lim_{j\rightarrow\infty}\sup\{h_{k}(x)-h_{j}(x):x\in\lbrack0,1]\}\\
= & 1.
\end{array}
\]
Thus $I_{\{\langle B\rangle_{1}=\frac{1}{2}\}}\not \in L_{G}^{2}(\Omega_{1})$.
Similarly, we can prove that $I_{\{\langle B\rangle_{1}\geq\frac{1}{2}\}}$ and
$I_{\{\langle B\rangle_{1}\leq\frac{1}{2}\}}\not \in L_{G}^{2}(\Omega_{1})$,
which implies that
\[
I_{\{\langle B\rangle_{1}>\frac{1}{2}\}}=1-I_{\{\langle B\rangle_{1}\leq
\frac{1}{2}\}}\not \in L_{G}^{2}(\Omega_{1}).
\]
Set $\mathcal{H}_{1}=\{h_{k}(\langle B\rangle_{1}):k\geq1\}$ and
$\mathcal{H}_{2}=\{g_{k}(\langle B\rangle_{1}):k\geq1\}$, where
\[
g_{k}(x)=k(x-\frac{1}{2}+\frac{1}{k})I_{[\frac{1}{2}-\frac{1}{k},\frac{1}{2}%
]}(x)+I_{(\frac{1}{2},\infty)}(x).
\]
We assert that either $\underset{\xi\in\mathcal{H}_{1}}{ess\inf}\xi$ or
$\underset{\xi\in\mathcal{H}_{2}}{ess\inf}\xi$ does not exist. Otherwise,
$\underset{\xi\in\mathcal{H}_{1}}{ess\inf}\xi$ and $\underset{\xi
\in\mathcal{H}_{2}}{ess\inf}\xi$ belong to $L_{G}^{2}(\Omega_{1})$.

By the definition we get%
\[%
\begin{array}
[c]{l}%
\underset{\xi\in\mathcal{H}_{1}}{ess\inf}\xi\leq I_{\{\langle B\rangle
_{1}=\frac{1}{2}\}\text{ }}\text{q.s.};\\
\underset{\xi\in\mathcal{H}_{2}}{ess\inf}\xi\leq I_{\{\langle B\rangle_{1}%
\geq\frac{1}{2}\}}\text{ q.s.};\\
\underset{\xi\in\mathcal{H}_{1}}{(ess\inf}\xi-\underset{\xi\in\mathcal{H}%
_{2}}{ess\inf}\xi)^{+}\in L_{G}^{2}(\Omega_{1});\\
\underset{\xi\in\mathcal{H}_{1}}{(ess\inf}\xi-\underset{\xi\in\mathcal{H}%
_{2}}{ess\inf}\xi)^{-}\in L_{G}^{2}(\Omega_{1}),
\end{array}
\]
which implies that
\[
I_{\{\langle B\rangle_{1}=\frac{1}{2}\}}\underset{\xi\in\mathcal{H}%
_{2}}{ess\inf}\xi=\underset{\xi\in\mathcal{H}_{1}}{ess\inf}\xi.
\]
Note that $\tilde{h}_{k}(\langle B\rangle_{1})\leq\underset{\xi\in
\mathcal{H}_{2}}{ess\inf}\xi$ for $k\geq1$, where
\[
\tilde{h}_{k}(x)=k(x-\frac{1}{2})I_{[\frac{1}{2},\frac{1}{2}+\frac{1}{k}%
)}(x)+I_{[\frac{1}{2}+\frac{1}{k},\infty)}(x).
\]
It yields that $I_{\{\langle B\rangle_{1}>\frac{1}{2}\}}\leq\underset{\xi
\in\mathcal{H}_{2}}{ess\inf}\xi$ q.s.. Then $\underset{\xi\in\mathcal{H}%
_{2}}{ess\inf}\xi=\underset{\xi\in\mathcal{H}_{1}}{ess\inf}\xi+I_{\{\langle
B\rangle_{1}>\frac{1}{2}\}}$ which implies $I_{\{\langle B\rangle_{1}>\frac
{1}{2}\}}\in L_{G}^{2}(\Omega_{1})$.

But this contradicts to $I_{\{\langle B\rangle_{1}>\frac{1}{2}\}}%
\not \in L_{G}^{2}(\Omega_{1})$. $\blacksquare$
\end{example}

Our stochastic optimal control problem is: for given $x\in\mathbb{R}^{n}$, to
find $u(\cdot)\in\mathcal{U}[t,T]$ so as to minimize the objective function
$Y_{t}^{t,x,u}$.

For $x\in\mathbb{R}^{n}$, we define the value function%
\begin{equation}
V(t,x):=\underset{u\in\mathcal{U}[t,T]}{ess\inf}Y_{t}^{t,x,u}\text{ for }%
x\in\mathbb{R}^{n}. \label{valuefunction0}%
\end{equation}

In the following we will prove that $V(\cdot,\cdot)$ exists and is
deterministic and for each $\xi\in\cup_{\varepsilon>0}L_{G}^{2+\varepsilon
}(\Omega_{t};\mathbb{R}^{n})$, $V(t,\xi)=\underset{u\in\mathcal{U}%
[t,T]}{ess\inf}Y_{t}^{t,\xi,u}$. Futhermore, we will obtain the dynamic
programming principle and the related fully nonlinear HJB equation.

\section{Properties of the value function}

We first give some notations:%
\begin{align*}
L_{ip}(\Omega_{s}^{t})  &  :=\{ \varphi(B_{t_{1}}-B_{t},...,B_{t_{n}}%
-B_{t}):n\geq1,t_{1},...,t_{n}\in\lbrack t,s],\varphi\in C_{b.Lip}%
(\mathbb{R}^{d\times n})\};\\
L_{G}^{2}(\Omega_{s}^{t})  &  :=\{ \text{the completion of }L_{ip}(\Omega
_{s}^{t})\text{ under the norm }\Vert\cdot\Vert_{2,G}\};\\
M_{G}^{0,t}(t,T)  &  :=\{ \eta_{s}=\sum_{i=0}^{N-1}\xi_{i}I_{[t_{i},t_{i+1}%
)}(s):t=t_{0}<\cdots<t_{N}=T,\xi_{i}\in L_{ip}(\Omega_{t_{i}}^{t})\};\\
M_{G}^{2,t}(t,T)  &  :=\{ \text{the completion of }M_{G}^{0,t}(t,T)\text{
under the norm }\Vert\cdot\Vert_{M_{G}^{2}}\};\\
\mathcal{U}[t,T]  &  :=\{u:u\in M_{G}^{2}(t,T;\mathbb{R}^{m})\text{ with
values in }U\};\\
\mathcal{U}^{t}[t,T]  &  :=\{u:u\in M_{G}^{2,t}(t,T;\mathbb{R}^{m})\text{ with
values in }U\};\\
\mathbb{U}[t,T]  &  :=\{u=\sum\limits_{i=1}^{n}I_{A_{i}}u^{i}:n\in
\mathbb{N},u^{i}\in\mathcal{U}^{t}[t,T],I_{A_{i}}\in L_{G}^{2}(\Omega
_{t}),\Omega=%
{\displaystyle\bigcup\limits_{i=1}^{n}}
A_{i}\};\\
\mathbb{U}^{t}[t,T]  &  :=\{u=\sum_{i=0}^{N-1}(\sum_{j=1}^{l_{i}}a_{j}%
^{i}I_{A_{j}^{i}})I_{[t_{i},t_{i+1})}(s):l_{i}\in\mathbb{N},a_{j}^{i}\in
U,I_{A_{j}^{i}}\in L_{G}^{2}(\Omega_{t_{i}}^{t}),\Omega=%
{\displaystyle\bigcup\limits_{j=1}^{l_{i}}}
A_{j}^{i}\}.
\end{align*}

\begin{remark}
\label{quasi1}For $t=t_{0}<\cdots<t_{N}=T$, $\xi_{i}\in L_{G}^{2}%
(\Omega_{t_{i}}^{t})$, it is easy to check that $\sum_{i=0}^{N-1}\xi
_{i}I_{[t_{i},t_{i+1})}(s)\in M_{G}^{2,t}(t,T)$. From this we can deduce that
$\mathbb{U}^{t}[t,T]\subset\mathcal{U}^{t}[t,T]\subset\mathbb{U}%
[t,T]\subset\mathcal{U}[t,T]$.
\end{remark}

In order to prove%
\[
V(t,x)=\inf_{u\in\mathcal{U}^{t}[t,T]}Y_{t}^{t,x,u}=\inf_{u\in\mathbb{U}%
^{t}[t,T]}Y_{t}^{t,x,u},
\]
we need the following lemmas.

\begin{lemma}
\label{le-dpp-2} Let $u\in\mathcal{U}[t,T]$ be given. Then there exists a
sequence $(u^{k})_{k\geq1}$ in $\mathbb{U}[t,T]$ such that%
\[
\lim_{k\rightarrow\infty}\mathbb{\hat{E}}[\int_{t}^{T}|u_{s}-u_{s}^{k}%
|^{2}ds]=0.
\]

\end{lemma}

\begin{proof}
For each $\varepsilon>0$, we only need to prove that there exists a process
$v\in\mathbb{U}[t,T]$ such that $\mathbb{\hat{E}}[\int_{t}^{T}|u_{s}%
-v_{s}|^{2}ds]\leq\varepsilon$. Since $u\in M_{G}^{2}(t,T;\mathbb{R}^{m})$,
there exists a sequence processes $v^{k}\in M_{G}^{0}(t,T;\mathbb{R}^{m})$
such that $\mathbb{\hat{E}}[\int_{t}^{T}|u_{s}-v_{s}^{k}|^{2}ds]\rightarrow0$.
Set $U^{\varepsilon}:=\{a\in\mathbb{R}^{m}:d(a,U)\leq\frac{\sqrt{\varepsilon}%
}{4}\}$, then%
\[
\mathbb{\hat{E}}[\int_{t}^{T}|u_{s}-v_{s}^{k}|^{2}ds]\geq\mathbb{\hat{E}}%
[\int_{t}^{T}|u_{s}-v_{s}^{k}|^{2}I_{\{v_{s}^{k}\not \in U^{\varepsilon}%
\}}ds]\geq\frac{\varepsilon}{16}\mathbb{\hat{E}}[\int_{t}^{T}I_{\{v_{s}%
^{k}\not \in U^{\varepsilon}\}}ds],
\]
which implies that $\mathbb{\hat{E}}[\int_{t}^{T}I_{\{v_{s}^{k}\not \in
U^{\varepsilon}\}}ds]\rightarrow0$. Thus there exists a $k_{0}\geq1$ such
that
\[
\mathbb{\hat{E}}[\int_{t}^{T}|u_{s}-v_{s}^{k_{0}}|^{2}ds]\leq\frac
{\varepsilon}{4},\text{ }\mathbb{\hat{E}}[\int_{t}^{T}I_{\{v_{s}^{k_{0}%
}\not \in U^{\varepsilon}\}}ds]\leq\frac{\varepsilon}{16M^{2}},
\]
where $M=\sup\{|a|:a\in U\}$. Set $\bar{v}=v^{k_{0}}$, we can write $\bar{v}$
as%
\[
\bar{v}_{s}=\varphi_{0}(\xi_{0})I_{[t_{0},t_{1})}(s)+\sum_{i=1}^{N-1}%
\varphi_{i}(\xi_{0},\xi_{i})I_{[t_{i},t_{i+1})}(s),
\]
where $t=t_{0}<t_{1}<\cdots<t_{N}=T$, $\xi_{0}=(B_{s_{1}^{0}},\ldots
,B_{s_{k_{0}}^{0}})$ for $s_{j}^{0}\in\lbrack0,t]$, $\xi_{i}=(B_{s_{1}^{i}%
}-B_{t},\ldots,B_{s_{k_{i}}^{i}}-B_{t})$ for $s_{j}^{i}\in\lbrack t,t_{i}]$,
$i\geq1$, $\varphi_{i}\in C_{b.Lip}(\mathbb{R}^{n_{i}};\mathbb{R}^{m})$ with
$n_{0}=dk_{0}$, $n_{i}=d(k_{0}+k_{i})$ for $i\geq1$. Obviously, we can find
two constants $\bar{M}>0$ and $L>0$ such that for $i\leq N-1$,%
\[
|\varphi_{i}|\leq\bar{M},\text{ }|\varphi_{i}(x^{i})-\varphi_{i}(\bar{x}%
^{i})|\leq L|x^{i}-\bar{x}^{i}|\text{ for }x^{i},\bar{x}^{i}\in\mathbb{R}%
^{n_{i}}.
\]
For each $k\geq1$, we can find finite nonempty cubes $A_{j}^{i,k}%
\subset\mathbb{R}^{dk_{i}}$, $i\geq0$, $j=1,\ldots,l_{i}^{k}-1$, such that
$[-ke^{i},ke^{i})=\cup_{j\leq l_{i}^{k}-1}A_{j}^{i,k}$ with $e^{i}%
=[1,\ldots,1]^{T}\in\mathbb{R}^{dk_{i}}$ and $\rho(A_{j}^{i,k}):=\sup
\{|x^{i}-\bar{x}^{i}|:x^{i},\bar{x}^{i}\in A_{j}^{i,k}\} \leq\frac{1}{k}$. Set
$A_{l_{i}^{k}}^{i,k}=\mathbb{R}^{dk_{i}}\backslash\lbrack-ke^{i},ke^{i})$ and%
\begin{align*}
\bar{v}_{s}^{k} =  &  (\sum_{j_{0}\leq l_{0}^{k}}\varphi_{0}(x_{j_{0}}%
^{0,k})I_{\{ \xi_{0}\in A_{j_{0}}^{0,k}\}})I_{[t_{0},t_{1})}(s)\\
&  +\sum_{i=1}^{N-1}(\sum_{j_{0}\leq l_{0}^{k},j_{i}\leq l_{i}^{k}}\varphi
_{i}(x_{j_{0}}^{0,k},x_{j_{i}}^{i,k})I_{\{ \xi_{0}\in A_{j_{0}}^{0,k}\}}I_{\{
\xi_{i}\in A_{j_{i}}^{i,k}\}})I_{[t_{i},t_{i+1})}(s)\\
=  &  \sum_{j_{0}\leq l_{0}^{k}}I_{\{ \xi_{0}\in A_{j_{0}}^{0,k}\}}\left(
\varphi_{0}(x_{j_{0}}^{0,k})I_{[t_{0},t_{1})}(s)+\sum_{i=1}^{N-1}(\sum
_{j_{i}\leq l_{i}^{k}}\varphi_{i}(x_{j_{0}}^{0,k},x_{j_{i}}^{i,k})I_{\{
\xi_{i}\in A_{j_{i}}^{i,k}\}})I_{[t_{i},t_{i+1})}(s)\right)  ,
\end{align*}
where $x_{j}^{i,k}$ is one point belonging to $A_{j}^{i,k}$ for $i\geq0$ and
$j\leq l_{i}^{k}$. By Theorem \ref{quasi} we can get $I_{\{ \xi_{0}\in
A_{j_{0}}^{0,k}\}}\in L_{G}^{2}(\Omega_{t})$ and $I_{\{ \xi_{i}\in A_{j_{i}%
}^{i,k}\}}\in L_{G}^{2}(\Omega_{t_{i}}^{t})$ for $i\geq1$. Then we have%
\begin{align*}
\mathbb{\hat{E}}[\int_{t}^{T}|\bar{v}_{s}-\bar{v}_{s}^{k}|^{2}ds]  &  \leq
\sum_{i=0}^{N-1}\mathbb{\hat{E}}[|\bar{v}_{t_{i}}-\bar{v}_{t_{i}}^{k}%
|^{2}](t_{i+1}-t_{i})\\
&  \leq\sum_{i=0}^{N-1}\mathbb{\hat{E}}[\frac{2L^{2}}{k^{2}}+\frac{4\bar
{M}^{2}}{k^{2}}(|\xi_{0}|^{2}+|\xi_{i}|^{2})](t_{i+1}-t_{i})\\
&  \rightarrow0\text{ as }k\rightarrow\infty.
\end{align*}
We set%
\[
\tilde{v}_{s}^{k}=\sum_{j_{0}\leq l_{0}^{k}}I_{\{ \xi_{0}\in A_{j_{0}}%
^{0,k}\}}\left(  \tilde{\varphi}_{0}(x_{j_{0}}^{0,k})I_{[t_{0},t_{1})}%
(s)+\sum_{i=1}^{N-1}(\sum_{j_{i}\leq l_{i}^{k}}\tilde{\varphi}_{i}(x_{j_{0}%
}^{0,k},x_{j_{i}}^{i,k})I_{\{ \xi_{i}\in A_{j_{i}}^{i,k}\}})I_{[t_{i}%
,t_{i+1})}(s)\right)  ,
\]
where $\tilde{\varphi}_{i}(x_{j_{0}}^{0,k},x_{j_{i}}^{i,k})$ is one point in
$U$ such that $|\varphi_{i}(x_{j_{0}}^{0,k},x_{j_{i}}^{i,k})-\tilde{\varphi
}_{i}(x_{j_{0}}^{0,k},x_{j_{i}}^{i,k})|=d(\varphi_{i}(x_{j_{0}}^{0,k}%
,x_{j_{i}}^{i,k}),U)$. By Remark \ref{quasi1}, it is easy to verify that
$\tilde{v}^{k}\in\mathbb{U}[t,T]$ and%
\[
\tilde{\varphi}_{0}(x_{j_{0}}^{0,k})I_{[t_{0},t_{1})}(s)+\sum_{i=1}^{N-1}%
(\sum_{j_{i}\leq l_{i}^{k}}\tilde{\varphi}_{i}(x_{j_{0}}^{0,k},x_{j_{i}}%
^{i,k})I_{\{ \xi_{i}\in A_{j_{i}}^{i,k}\}})I_{[t_{i},t_{i+1})}(s)\in
\mathbb{U}^{t}[t,T].
\]
Note that%
\begin{align*}
|u_{s}-\tilde{v}_{s}^{k}|^{2}  &  =|u_{s}-\tilde{v}_{s}^{k}|^{2}I_{\{ \bar
{v}_{s}\not \in U^{\varepsilon}\}}+|u_{s}-\tilde{v}_{s}^{k}|^{2}I_{\{ \bar
{v}_{s}\in U^{\varepsilon}\}}\\
&  \leq4M^{2}I_{\{ \bar{v}_{s}\not \in U^{\varepsilon}\}}+2|u_{s}-\bar{v}%
_{s}|^{2}+2|\bar{v}_{s}-\tilde{v}_{s}^{k}|^{2}I_{\{ \bar{v}_{s}\in
U^{\varepsilon}\}}\\
&  \leq4M^{2}I_{\{ \bar{v}_{s}\not \in U^{\varepsilon}\}}+2|u_{s}-\bar{v}%
_{s}|^{2}+2(\frac{\sqrt{\varepsilon}}{4}+\frac{2\sqrt{2}L}{k})^{2},
\end{align*}
then we get%
\[
\underset{k\rightarrow\infty}{\limsup}\mathbb{\hat{E}}[\int_{t}^{T}%
|u_{s}-\tilde{v}_{s}^{k}|^{2}ds]\leq\frac{7}{8}\varepsilon.
\]
Thus there exists a $k_{1}\geq1$ such that $\mathbb{\hat{E}}[\int_{t}%
^{T}|u_{s}-\tilde{v}_{s}^{k_{1}}|^{2}ds]\leq\varepsilon$. The proof is
complete by taking $v=\tilde{v}^{k_{1}}\in\mathbb{U}[t,T]$.
\end{proof}

\begin{lemma}
\label{le-dpp-n2}Let $u\in\mathcal{U}^{t}[t,T]$ be given. Then there exists a
sequence $(u^{k})_{k\geq1}$ in $\mathbb{U}^{t}[t,T]$ such that%
\[
\lim_{k\rightarrow\infty}\mathbb{\hat{E}}[\int_{t}^{T}|u_{s}-u_{s}^{k}%
|^{2}ds]=0.
\]

\end{lemma}

\begin{proof}
The proof is the same as Lemma \ref{le-dpp-2}, we omit it.
\end{proof}

\begin{lemma}
\label{le-dpp-3} Let $\xi\in\cup_{\varepsilon>0}L_{G}^{2+\varepsilon}%
(\Omega_{t};\mathbb{R}^{n})$, $u\in\mathcal{U}[t,T]$ and $v_{s}=\sum_{i=1}%
^{N}I_{A_{i}}v_{s}^{i}\in\mathbb{U}[t,T]$. Then there exists a constant
$L_{1}$ depending on $T$, $G$ and $C$ such that
\[
\mathbb{\hat{E}}[|Y_{t}^{t,\xi,u}-\sum_{i=1}^{N}I_{A_{i}}Y_{t}^{t,\xi,v^{i}%
}|^{2}]\leq L_{1}\mathbb{\hat{E}}[\int_{t}^{T}|u_{s}-v_{s}|^{2}ds].
\]

\end{lemma}

\begin{proof}
Consider the following equations:%
\begin{align*}
dX_{s}^{t,\xi,v^{i}}  &  =b(s,X_{s}^{t,\xi,v^{i}},v_{s}^{i})ds+h_{ij}%
(s,X_{s}^{t,\xi,v^{i}},v_{s}^{i})d\langle B^{i},B^{j}\rangle_{s}%
+\sigma(s,X_{s}^{t,\xi,v^{i}},v_{s}^{i})dB_{s},\\
X_{t}^{t,\xi,v^{i}}  &  =\xi,\text{\ }s\in\lbrack t,T]\text{, }i=1,...,N.
\end{align*}%
\begin{align*}
-dY_{s}^{t,\xi,v^{i}}  &  =f(s,X_{s}^{t,\xi,v^{i}},Y_{s}^{t,\xi,v^{i}}%
,Z_{s}^{t,\xi,v^{i}},v_{s}^{i})ds+g_{ij}(s,X_{s}^{t,\xi,v^{i}},Y_{s}%
^{t,\xi,v^{i}},Z_{s}^{t,\xi,v^{i}},v_{s}^{i})d\langle B^{i},B^{j}\rangle_{s}\\
&  -Z_{s}^{t,\xi,v^{i}}dB_{s}-dK_{s}^{t,\xi,v^{i}},\\
Y_{T}^{t,\xi,v^{i}}  &  =\Phi(X_{T}^{t,\xi,v^{i}}),\text{ \ \ \ }s\in\lbrack
t,T]\text{, }i=1,...,N.
\end{align*}
For $s\in\lbrack t,T]$, we set
\[
\bar{X}_{s}^{t,\xi,v}=\sum_{i=1}^{N}I_{A_{i}}X_{s}^{t,\xi,v^{i}},\bar{Y}%
_{s}^{t,\xi,v}=\sum_{i=1}^{N}I_{A_{i}}Y_{s}^{t,\xi,v^{i}},\bar{Z}_{s}%
^{t,\xi,v}=\sum_{i=1}^{N}I_{A_{i}}Z_{s}^{t,\xi,v^{i}},\bar{K}_{s}^{t,\xi
,v}=\sum_{i=1}^{N}I_{A_{i}}K_{s}^{t,\xi,v^{i}}.
\]
Multiplying $I_{A_{i}}$ on both sides of the above equations and summing up,
we have%
\begin{align*}
d\bar{X}_{s}^{t,\xi,v}  &  =b(s,\bar{X}_{s}^{t,\xi,v},v_{s})ds+h_{ij}%
(s,\bar{X}_{s}^{t,\xi,v},v_{s})d\langle B^{i},B^{j}\rangle_{s}+\sigma
(s,\bar{X}_{s}^{t,\xi,v},v_{s})dB_{s},\\
\bar{X}_{t}^{t,\xi,v}  &  =\xi,\text{ \ }s\in\lbrack t,T],
\end{align*}%
\begin{align*}
-d\bar{Y}_{s}^{t,\xi,v}  &  =f(s,\bar{X}_{s}^{t,\xi,v},\bar{Y}_{s}^{t,\xi
,v},\bar{Z}_{s}^{t,\xi,v},v_{s})ds+g_{ij}(s,\bar{X}_{s}^{t,\xi,v},\bar{Y}%
_{s}^{t,\xi,v},\bar{Z}_{s}^{t,\xi,v},v_{s})d\langle B^{i},B^{j}\rangle_{s}\\
&  -Z_{s}^{t,\xi,v^{i}}dB_{s}-d\bar{K}_{s}^{t,\xi,v},\\
\bar{Y}_{T}^{t,\xi,v}  &  =\Phi(\bar{X}_{T}^{t,\xi,v}),\text{ \ \ \ }%
s\in\lbrack t,T].
\end{align*}
By Theorem \ref{BSDE-est}, we can obtain that there exists a constant
$C_{1}>0$ depending on $T$, $G$ and $C$ such that%
\begin{equation}%
\begin{array}
[c]{rl}
& |Y_{t}^{t,\xi,u}-\sum_{i=1}^{N}I_{A_{i}}Y_{t}^{t,\xi,v^{i}}|^{2}%
=|Y_{t}^{t,\xi,u}-\bar{Y}_{t}^{t,\xi,v}|^{2}\\
\leq & C_{1}\mathbb{\hat{E}}_{t}[\mid\Phi(X_{T}^{t,\xi,u})-\Phi(\bar{X}%
_{T}^{t,\xi,v})\mid^{2}+\int\nolimits_{t}^{T}(\mid X_{s}^{t,\xi,u}-\bar{X}%
_{s}^{t,\xi,v}\mid^{2}+\mid u_{s}-v_{s}\mid^{2})ds]\\
\leq & C_{1}\mathbb{\hat{E}}_{t}[C\mid X_{T}^{t,\xi,u}-\bar{X}_{T}^{t,\xi
,v}\mid^{2}+\int\nolimits_{t}^{T}(\mid X_{s}^{t,\xi,u}-\bar{X}_{s}^{t,\xi
,v}\mid^{2}+\mid u_{s}-v_{s}\mid^{2})ds],
\end{array}
\label{estimation-1}%
\end{equation}
where $C$ is the Lipschitz constant of $\Phi$. By Theorem \ref{SDE-est}, there
exists a constant $C_{2}>0$ depending on $T$, $n$, $G$ and $C$ such that%
\[
\mathbb{\hat{E}}[\mid X_{s}^{t,\xi,u}-\bar{X}_{s}^{t,\xi,v}\mid^{2}]\leq
C_{2}\mathbb{\hat{E}}[\int\nolimits_{t}^{T}\mid u_{s}-v_{s}\mid^{2}ds].
\]
Taking $G$-expectation on both sides of (\ref{estimation-1}), we obtain the result.
\end{proof}

\begin{remark}
By the definition of generalized conditional $G$-expectation and Proposition
2.5 in \cite{HJPS1}, the above analysis still holds for the case that
$\{A_{i}\}_{i=1}^{N}$ is a $\mathcal{B}(\Omega_{t})$-partition of $\Omega$.
\end{remark}

\begin{theorem}
\label{the-dpp-1} The value function $V(t,x)$ exists and
\[
V(t,x)=\inf_{u\in\mathcal{U}^{t}[t,T]}Y_{t}^{t,x,u}=\inf_{u\in\mathbb{U}%
^{t}[t,T]}Y_{t}^{t,x,u}.
\]

\end{theorem}

\begin{proof}
For each $v\in\mathcal{U}^{t}[t,T]$, it is easy to check that $Y_{t}^{t,x,v}$
is a constant. In the following, we prove that $\underset{u\in\mathcal{U}%
[t,T]}{ess\inf}Y_{t}^{t,x,u}=\inf_{v\in\mathcal{U}^{t}[t,T]}Y_{t}^{t,x,v}%
,\;$q.s.. Since $\mathcal{U}^{t}[t,T]\subset\mathcal{U}[t,T]$, we only need to
show that $Y_{t}^{t,x,u}\geq\inf_{v\in\mathcal{U}^{t}[t,T]}Y_{t}^{t,x,v}$ q.s.
for each $u\in\mathcal{U}[t,T]$. For each fixed $u\in\mathcal{U}[t,T],$ by
Lemma \ref{le-dpp-2}, there exists a sequence $u^{k}=\sum_{i=1}^{N_{k}%
}I_{A_{i}^{k}}v^{i,k}\in\mathbb{U}[t,T]$, $k=1,2,...,$ such that
\[
\mathbb{\hat{E}}[\int_{t}^{T}|u_{s}-\sum_{i=1}^{N_{k}}I_{A_{i}^{k}}%
v^{i,k}|^{2}ds]\rightarrow0.
\]
By Lemma \ref{le-dpp-3},
\[%
\begin{array}
[c]{rl}
& \mathbb{\hat{E}}[\mid Y_{t}^{t,x,u}-\sum_{i=1}^{N_{k}}I_{A_{i}^{k}}%
Y_{t}^{t,x,v^{i,k}}\mid^{2}]\\
\leq & L_{1}\mathbb{\hat{E}}[\int_{t}^{T}|u_{s}-\sum_{i=1}^{N_{k}}I_{A_{i}%
^{k}}v^{i,k}|^{2}ds].
\end{array}
\]
It yields that $\sum_{i=1}^{N_{k}}I_{A_{i}^{k}}Y_{t}^{t,x,v^{i,k}}$ converges
to $Y_{t}^{t,x,u}$ in $\mathbb{L}_{G}^{2}$. Then there exists a subsequence
(for simplicity, we still denote it by $\{ \sum_{i=1}^{N_{k}}I_{A_{i}^{k}%
}Y_{t}^{t,x,v^{i,k}}\}$) which converges to $Y_{t}^{t,x,u}$ q.s.. Note that
\[
\sum_{i=1}^{N_{k}}I_{A_{i}^{k}}Y_{t}^{t,x,v^{i,k}}\geq\inf_{v\in
\mathcal{U}^{t}[t,T]}Y_{t}^{t,x,v}\text{ q.s.},
\]
then we have $Y_{t}^{t,x,u}\geq\inf_{v\in\mathcal{U}^{t}[t,T]}Y_{t}^{t,x,v}$
q.s.. Thus
\[
V(t,x)=\underset{u\in\mathcal{U}[t,T]}{ess\inf}Y_{t}^{t,x,u}=\inf
_{v\in\mathcal{U}^{t}[t,T]}Y_{t}^{t,x,v}.
\]
Similarly, by Lemmas \ref{le-dpp-n2} and \ref{le-dpp-3}, we can get
$\inf_{u\in\mathcal{U}^{t}[t,T]}Y_{t}^{t,x,u}=\inf_{u\in\mathbb{U}^{t}%
[t,T]}Y_{t}^{t,x,u}$. The proof is complete.
\end{proof}

\begin{lemma}
\label{le-dpp-4} There exists a constant $L_{2}>0$ depending on $T$, $G$ and
$C$ such that
\[
\mid V(t,x)-V(t,y)\mid\leq L_{2}\mid x-y\mid\text{ for any }x,y\in
\mathbb{R}^{n}.
\]

\end{lemma}

\begin{proof}
By Theorems \ref{SDE-est} and \ref{BSDE-est}, for any $x$, $y\in\mathbb{R}%
^{n}$ and $u\in\mathcal{U}^{t}[t,T]$,
\[%
\begin{array}
[c]{rl}
& \mid Y_{t}^{t,x,u}-Y_{t}^{t,y,u}\mid^{2}\\
\leq & C_{1}\mathbb{\hat{E}}[\mid\Phi(X_{T}^{t,x,u})-\Phi(X_{T}^{t,y,u}%
)\mid^{2}+\int\nolimits_{t}^{T}\mid X_{s}^{t,x,u}-X_{s}^{t,y,u}\mid^{2}ds]\\
\leq & C_{2}\mid x-y\mid^{2}.
\end{array}
\]
It is easy to verify that $|\inf_{v\in\mathcal{U}^{t}[t,T]}Y_{t}^{t,x,v}%
-\inf_{v\in\mathcal{U}^{t}[t,T]}Y_{t}^{t,y,v}|\leq\sup_{v\in\mathcal{U}%
^{t}[t,T]}|Y_{t}^{t,x,v}-Y_{t}^{t,y,v}|$. Thus by the above estimate and
Theorem \ref{the-dpp-1}, we obtain the result.
\end{proof}

\begin{lemma}
\label{le-neww-dpp-5}There exists a constant $L_{3}>0$ depending on $T$, $G$
and $C$ such that
\[
\mid V(t,x)\mid\leq L_{3}(1+\mid x\mid)\text{ for any }x\in\mathbb{R}^{n}.
\]

\end{lemma}

\begin{proof}
The proof is similar to Lemma \ref{le-dpp-4}, we omit it.
\end{proof}

\begin{theorem}
\label{the-dpp-1-1} For any $\xi\in\cup_{\varepsilon>0}L_{G}^{2+\varepsilon
}(\Omega_{t};\mathbb{R}^{n})$, we have%
\[
V(t,\xi)=\underset{u\in\mathcal{U}[t,T]}{ess\inf}Y_{t}^{t,\xi,u}.
\]

\end{theorem}

\begin{proof}
First, we prove that $\forall u\in\mathcal{U}[t,T]$, $V(t,\xi)\leq
Y_{t}^{t,\xi,u}\;$q.s..

For a fixed $\xi\in\cup_{\varepsilon>0}L_{G}^{2+\varepsilon}(\Omega
_{t};\mathbb{R}^{n})$, we can find a sequence $\xi^{k}=\sum_{i=1}^{N_{k}}%
x_{i}^{k}I_{A_{i}^{k}}$, $k=1,2,...,$where $x_{i}^{k}\in\mathbb{R}^{n}$ and
$\{A_{i}^{k}\}_{i=1}^{N_{k}}$ is a $\mathcal{B}(\Omega_{t})$-partition of
$\Omega$, such that
\[
\lim_{k\rightarrow\infty}\mathbb{\hat{E}}[|\xi-\xi^{k}|^{2}]=0.
\]
Here $I_{A_{i}^{k}}$ may not in $L_{G}^{2}(\Omega_{t})$. By Lemma
\ref{le-dpp-4}, we have%
\[
\mathbb{\hat{E}}[|V(t,\xi)-V(t,\xi^{k})|^{2}]\leq L_{2}^{2}\mathbb{\hat{E}%
}[|\xi-\xi^{k}|^{2}]\rightarrow0.
\]
By similar analysis as in Lemma \ref{le-dpp-3} and the definition of
generalized conditional $G$-expectation,
\[%
\begin{array}
[c]{rl}
& |Y_{t}^{t,\xi,u}-\sum_{i=1}^{N_{k}}I_{A_{i}^{k}}Y_{t}^{t,x_{i}^{k},u}|^{2}\\
\leq & C_{1}\mathbb{\hat{E}}_{t}[\mid\Phi(X_{T}^{t,\xi,u})-\Phi(\sum
_{i=1}^{N_{k}}I_{A_{i}^{k}}X_{T}^{t,x_{i}^{k},u})\mid^{2}+\int\nolimits_{t}%
^{T}\mid X_{s}^{t,\xi,u}-\sum_{i=1}^{N_{k}}I_{A_{i}^{k}}X_{s}^{t,x_{i}^{k}%
,u}\mid^{2}ds]\\
\leq & C_{2}\mathbb{\hat{E}}_{t}[\mid\xi-\xi^{k}\mid^{2}].
\end{array}
\]
Then
\[
\lim_{k\rightarrow\infty}\mathbb{\hat{E}[}|Y_{t}^{t,\xi,u}-\sum_{i=1}^{N_{k}%
}I_{A_{i}^{k}}Y_{t}^{t,x_{i}^{k},u}|^{2}]=0.
\]
Note that%
\[
V(t,\xi^{k})=\sum_{i=1}^{N_{k}}I_{A_{i}^{k}}V(t,x_{i}^{k})\leq\sum
_{i=1}^{N_{k}}I_{A_{i}^{k}}Y_{t}^{t,x_{i}^{k},u}\; \text{q.s.}.
\]
Thus
\[
V(t,\xi)\leq Y_{t}^{t,\xi,u}\; \text{q.s.}.
\]

Second, we prove that for a given $\eta\in L_{G}^{2}(\Omega_{t})$, if $\forall
u\in\mathcal{U}[t,T]$, $\eta\leq Y_{t}^{t,\xi,u}\;$q.s., then $\eta\leq
V(t,\xi)\;$q.s..

By the above analysis, we know that%
\[
|Y_{t}^{t,\xi,u}-\sum_{i=1}^{N_{k}}I_{A_{i}^{k}}Y_{t}^{t,x_{i}^{k},u}|^{2}\leq
C_{2}\mathbb{\hat{E}}_{t}[\mid\xi-\xi^{k}\mid^{2}]\; \; \text{q.s.}.
\]
Then, for any $u\in\mathcal{U}[t,T]$,%
\[
\eta\leq\sum_{i=1}^{N_{k}}I_{A_{i}^{k}}Y_{t}^{t,x_{i}^{k},u}+\sqrt
{C_{2}\mathbb{\hat{E}}_{t}[\mid\xi-\xi^{k}\mid^{2}]}\; \; \text{q.s.}.
\]
For each fixed $N_{k}$, it yields that%
\[%
\begin{array}
[c]{rl}%
\eta & \leq\sum_{i=1}^{N_{k}}I_{A_{i}^{k}}V(t,x_{i}^{k})+\sqrt{C_{2}%
\mathbb{\hat{E}}_{t}[\mid\xi-\xi^{k}\mid^{2}]}\\
& =V(t,\xi^{k})+\sqrt{C_{2}\mathbb{\hat{E}}_{t}[\mid\xi-\xi^{k}\mid^{2}]}\; \;
\text{q.s.}.
\end{array}
\]
Note that
\[
\lim_{k\rightarrow\infty}\mathbb{\hat{E}}[|V(t,\xi)-V(t,\xi^{k})|^{2}]=0
\]
and%
\[
\lim_{k\rightarrow\infty}\mathbb{\hat{E}}[|\xi-\xi^{k}|^{2}]=0.
\]
Then there exists a subsequence $(\xi^{k_{i}})$\ of $(\xi^{k})$ such that as
$k_{i}\longrightarrow\infty$,%
\[
V(t,\xi^{k_{i}})\longrightarrow V(t,\xi),\; \mathbb{\hat{E}}_{t}[\mid\xi
-\xi^{k_{i}}\mid^{2}]\longrightarrow0,\; \; \text{q.s.}.
\]
Thus $\eta\leq V(t,\xi)\; \;$q.s.. This completes the proof.
\end{proof}

\section{Dynamic programming principle}

For given initial data $(t,x)$, a positive real number $\delta\leq T-t$ and
$\eta\in\cup_{\varepsilon>0}L_{G}^{2+\varepsilon}(\Omega_{t+\delta})$, we
define
\[
\mathbb{G}_{t,t+\delta}^{t,x,u}[\eta]:=\tilde{Y}_{t}^{t,x,u},
\]
where $(X_{s}^{t,x,u},\tilde{Y}_{s}^{t,x,u},\tilde{Z}_{s}^{t,x,u})_{t\leq
s\leq t+\delta}$ is the solution of the following forward and backward
equations:%
\[%
\begin{array}
[c]{rl}%
dX_{s}^{t,x,u}= & b(s,X_{s}^{t,x,u},u_{s})ds+h_{ij}(s,X_{s}^{t,x,u}%
,u_{s})d\langle B^{i},B^{j}\rangle_{s}+\sigma(s,X_{s}^{t,x,u},u_{s})dB_{s},\\
X_{t}^{t,x,u}= & x,
\end{array}
\]

and%
\begin{equation}%
\begin{array}
[c]{ll}%
-d\tilde{Y}_{s}^{t,x,u}= & f(s,X_{s}^{t,x,u},\tilde{Y}_{s}^{t,x,u},\tilde
{Z}_{s}^{t,x,u},u_{s})ds+g_{ij}(s,X_{s}^{t,x,u},\tilde{Y}_{s}^{t,x,u}%
,\tilde{Z}_{s}^{t,x,u},u_{s})d\langle B^{i},B^{j}\rangle_{s}\\
& -\tilde{Z}_{s}^{t,x,u}dB_{s}-d\tilde{K}_{s}^{t,x,u},\\
\tilde{Y}_{t+\delta}^{t,x,u}= & \eta,\text{ \ \ \ }s\in\lbrack t,t+\delta
]\text{.}%
\end{array}
\label{state equ-2}%
\end{equation}
Note that $\mathbb{G}_{t,t+\delta}^{t,x,u}[\cdot]$ is a (backward) semigroup
which was first introduced by Peng in\ \cite{peng-dpp-1}.

Our main result in this section is the following\ dynamic programming principle.

\begin{theorem}
\label{the-dpp-2} Let Assumptions (A1) and (A2) hold. Then for any $t\leq
s\leq T$, $x\in\mathbb{R}^{n}$, we have
\begin{equation}%
\begin{array}
[c]{rl}%
V(t,x)= & \underset{u(\cdot)\in\mathcal{U}[t,s]}{\text{ess}\inf}%
\mathbb{G}_{t,s}^{t,x,u}[V(s,X_{s}^{t,x,u})]\\
= & \underset{u(\cdot)\in\mathcal{U}^{t}[t,s]}{\inf}\mathbb{G}_{t,s}%
^{t,x,u}[V(s,X_{s}^{t,x,u})].
\end{array}
\label{DPP}%
\end{equation}

\end{theorem}

In order to prove Theorem \ref{the-dpp-2}, we need the following lemmas.

\begin{lemma}
\label{le-new-dpp1}Let Assumptions (A1) and (A2) hold. Assume further that
there exist constants $L>0$, $\Lambda>0$ such that $|b|\leq L$, $|h_{ij}|\leq
L$ and $|\sigma_{i}|\leq\Lambda$ for $i\leq n$ and $(t,x,u)\in\lbrack
0,T]\times\mathbb{R}^{n}\times U$, where $\sigma_{i}$ is the $i$-th row of
$\sigma$. Then for any $t<s\leq T$, $x\in\mathbb{R}^{n}$, we have%
\[
V(t,x)\leq\underset{u(\cdot)\in\mathbb{U}^{t}[t,s]}{\inf}\mathbb{G}%
_{t,s}^{t,x,u}[V(s,X_{s}^{t,x,u})].
\]

\end{lemma}

\begin{proof}
For each $\varepsilon>0$, there exists a $u(\cdot)\in\mathbb{U}^{t}[t,s]$ such
that%
\[
\mathbb{G}_{t,s}^{t,x,u}[V(s,X_{s}^{t,x,u})]-\varepsilon\leq\underset{v(\cdot
)\in\mathbb{U}^{t}[t,s]}{\inf}\mathbb{G}_{t,s}^{t,x,v}[V(s,X_{s}^{t,x,v})].
\]
We can write $u(\cdot)$ as%
\[
u_{r}=a_{1}^{0}I_{[t_{0},t_{1})}(r)+\sum_{i=1}^{N-1}(\sum_{j=1}^{l_{i}}%
a_{j}^{i}I_{A_{j}^{i}})I_{[t_{i},t_{i+1})}(r),
\]
where $t=t_{0}<t_{1}<\cdots<t_{N}=s$, $l_{i}\in\mathbb{N}$, $a_{j}^{i}\in U$,
$I_{A_{j}^{i}}\in L_{G}^{2}(\Omega_{t_{i}}^{t})$ and $\{A_{j}^{i}%
\}_{j=1}^{l_{i}}$ is a partition of $\Omega$. Consider the following SDE: for
any $v(\cdot)\in\mathbb{U}^{t}[t,s]$,
\[%
\begin{array}
[c]{rl}%
d\tilde{X}_{r}^{t,x,v}= & b(r,\tilde{X}_{r}^{t,x,v}-\tilde{X}_{r}%
e,v_{r})dr+h_{ij}(r,\tilde{X}_{r}^{t,x,v}-\tilde{X}_{r}e,v_{r})d\langle
B^{i},B^{j}\rangle_{r}\\
& +\sigma(r,\tilde{X}_{r}^{t,x,v}-\tilde{X}_{r}e,v_{r})dB_{r}+(\Lambda
+1)edB_{r}^{1},\\
d\tilde{X}_{r}= & (\Lambda+1)dB_{r}^{1},\\
\tilde{X}_{t}^{t,x,v}= & x,\text{ }\tilde{X}_{t}=0,\text{ }r\in\lbrack t,s],
\end{array}
\]
where $e=[1,\ldots,1]^{T}\in\mathbb{R}^{n}$. It is easy to verify that%
\[
\tilde{X}_{r}^{t,x,v}=X_{r}^{t,x,v}+(\Lambda+1)(B_{r}^{1}-B_{t}^{1})e,\text{
}\tilde{X}_{r}=(\Lambda+1)(B_{r}^{1}-B_{t}^{1}),\text{ }r\in\lbrack t,s]
\]
is the solution of the above SDE. We set%
\[
I_{i+1}=\{J_{i+1}=(j_{0},j_{1},\ldots,j_{i}):1\leq j_{k}\leq l_{k},0\leq k\leq
i\},
\]
where $l_{0}=1$, $i\leq N-1$. For each given $J_{i+1}=(j_{0},j_{1}%
,\ldots,j_{i})\in I_{i+1}$, we denote%
\[
\tilde{X}_{r}^{t,x,J_{i+1}}:=\tilde{X}_{r}^{t,x,u^{J_{i+1}}},\text{ }%
r\in\lbrack t,t_{i+1}],
\]
where $u_{r}^{J_{i+1}}=\sum_{k=0}^{i}a_{j_{k}}^{k}I_{[t_{k},t_{k+1})}(r)$ is
deterministic. We claim that%
\begin{equation}
\tilde{X}_{t_{i+1}}^{t,x,u}=\sum_{J_{i+1}\in I_{i+1}}(%
{\displaystyle\prod\limits_{k=1}^{i}}
I_{A_{j_{k}}^{k}})\tilde{X}_{t_{i+1}}^{t,x,J_{i+1}}\text{ for }0\leq i\leq
N-1, \label{eq-dpp-1}%
\end{equation}
where $\Pi_{k=1}^{0}I_{A_{j_{k}}^{k}}=1$. It is easy to check that the
equality (\ref{eq-dpp-1}) holds for $i=0$. Suppose that the equality
(\ref{eq-dpp-1}) holds for $i_{0}\geq0$, then by the similar analysis as in
the proof of Lemma \ref{le-dpp-3}, we can get%
\[
\tilde{X}_{t_{i_{0}+2}}^{t,x,u}=\sum_{j=1}^{l_{i_{0}+1}}I_{A_{j}^{i_{0}+1}%
}\tilde{X}_{t_{i_{0}+2}}^{t_{i_{0}+1},\xi,a_{j}^{i_{0}+1}},\text{ }\tilde
{X}_{t_{i_{0}+2}}^{t_{i_{0}+1},\xi,a_{j}^{i_{0}+1}}=\sum_{J_{i_{0}+1}\in
I_{i_{0}+1}}(%
{\displaystyle\prod\limits_{k=1}^{i_{0}}}
I_{A_{j_{k}}^{k}})\tilde{X}_{t_{i_{0}+2}}^{t_{i_{0}+1},\xi^{J_{i_{0}+1}}%
,a_{j}^{i_{0}+1}},
\]
where $\xi=\tilde{X}_{t_{i_{0}+1}}^{t,x,u}$, $\xi^{J_{i_{0}+1}}=\tilde
{X}_{t_{i_{0}+1}}^{t,x,J_{i_{0}+1}}$. It is easy to verify that $\tilde
{X}_{t_{i_{0}+2}}^{t_{i_{0}+1},\xi^{J_{i_{0}+1}},a_{j}^{i_{0}+1}}=\tilde
{X}_{t_{i_{0}+2}}^{t,x,(J_{i_{0}+1},a_{j}^{i_{0}+1})}$. Thus the equality
(\ref{eq-dpp-1}) holds for $i_{0}+1$. From this we can deduce that%
\[
\tilde{X}_{s}^{t,x,u}=\sum_{J_{N}\in I_{N}}(%
{\displaystyle\prod\limits_{k=1}^{N-1}}
I_{A_{j_{k}}^{k}})\tilde{X}_{s}^{t,x,J_{N}}.
\]
It is easy to check that
\[
\sqrt{|\sigma_{i1}+\Lambda+1|^{2}+|\sigma_{i2}|^{2}+\cdots+|\sigma_{id}|^{2}%
}\geq1,
\]
where $\sigma_{ij}$ is the $i$-th row and $j$-th column of $\sigma$. Then by
Theorem \ref{quasi}, we have $I_{\{ \tilde{X}_{s}^{t,x,J_{N}}\in\lbrack
a,a^{\prime})\}}\in L_{G}^{2}(\Omega_{s}^{t})$ for $a$, $a^{\prime}%
\in\mathbb{R}^{n}$ with $a\leq a^{\prime}$. Thus%
\[
I_{\{ \tilde{X}_{s}^{t,x,u}\in\lbrack a,a^{\prime})\}}=\sum_{J_{N}\in I_{N}}(%
{\displaystyle\prod\limits_{k=1}^{N-1}}
I_{A_{j_{k}}^{k}})I_{\{ \tilde{X}_{s}^{t,x,J_{N}}\in\lbrack a,a^{\prime}%
)\}}\in L_{G}^{2}(\Omega_{s}^{t}).
\]
For each integer $k\geq1$, we can choose finite nonempty cubes $D_{j}%
^{k}\subset\mathbb{R}^{n}$, $E_{j}^{k}\subset\mathbb{R}$ and $x_{j}^{k}\in
D_{j}^{k}$, $q_{j}^{k}\in E_{j}^{k}$ for $j=1,\ldots,l^{k}$, such that
$[-ke,ke)=\cup_{j\leq l^{k}}D_{j}^{k}$, $[-k,k)=\cup_{j\leq l^{k}}E_{j}^{k}$,
$\rho(D_{j}^{k})=\sup\{|x-\bar{x}|:x,\bar{x}\in D_{j}^{k}\} \leq\frac{1}{k}$
and $\rho(E_{j}^{k})=\sup\{|q-\bar{q}|:q,\bar{q}\in E_{j}^{k}\} \leq\frac
{1}{\sqrt{n}(\Lambda+1)k}$. Set%
\[
\xi^{k,u}=\sum_{j_{1}=1}^{l^{k}}\sum_{j_{2}=1}^{l^{k}}(x_{j_{1}}^{k}%
-(\Lambda+1)q_{j_{2}}^{k}e)I_{\{ \tilde{X}_{s}^{t,x,u}\in D_{j_{1}}^{k}%
\}}I_{\{B_{s}^{1}-B_{t}^{1}\in E_{j_{2}}^{k}\}}.
\]
Note that $\tilde{X}_{s}^{t,x,u}=X_{s}^{t,x,u}+(\Lambda+1)(B_{s}^{1}-B_{t}%
^{1})e$, then we get%
\[
|X_{s}^{t,x,u}-\xi^{k,u}|\leq\frac{2}{k}+|X_{s}^{t,x,u}|\frac{|X_{s}%
^{t,x,u}|+\sqrt{n}(\Lambda+1)|B_{s}^{1}-B_{t}^{1}|}{k}.
\]
By Theorem \ref{SDE-est}, there exists a constant $C_{1}>0$ depending on $T$,
$n$, $U$, $G$ and $C$ such that%
\[
\mathbb{\hat{E}}[|X_{s}^{t,x,u}|^{4}]\leq C_{1}(1+|x|^{4}).
\]
Thus we obtain%
\begin{align*}
\mathbb{\hat{E}}[|X_{s}^{t,x,u}-\xi^{k,u}|^{2}]  &  \leq\mathbb{\hat{E}}%
[\frac{8(1+|X_{s}^{t,x,u}|^{4})+n^{2}(\Lambda+1)^{4}|B_{s}^{1}-B_{t}^{1}|^{4}%
}{k^{2}}]\\
&  \leq\frac{C_{2}(1+|x|^{4})}{k^{2}}.
\end{align*}
Set $\tilde{x}_{j_{1}j_{2}}^{k}=x_{j_{1}}^{k}-(\Lambda+1)q_{j_{2}}^{k}e$ for
$j_{1}$, $j_{2}\leq l^{k}$, by Theorem \ref{the-dpp-1}, $V(s,\tilde{x}%
_{j_{1}j_{2}}^{k})=\inf_{u\in\mathbb{U}^{s}[s,T]}Y_{s}^{s,\tilde{x}%
_{j_{1}j_{2}}^{k},u}$ and $V(s,0)=\inf_{u\in\mathbb{U}^{s}[s,T]}Y_{s}^{s,0,u}%
$, then we can choose $\bar{u}^{j_{1}j_{2},k}$, $\bar{u}^{0,k}\in
\mathbb{U}^{s}[s,T]$ such that
\[
V(s,\tilde{x}_{j_{1}j_{2}}^{k})\leq Y_{s}^{s,\tilde{x}_{j_{1}j_{2}}^{k}%
,\bar{u}^{j_{1}j_{2},k}}\leq V(s,\tilde{x}_{j_{1}j_{2}}^{k})+\varepsilon
,\text{ }V(s,0)\leq Y_{s}^{s,0,\bar{u}^{0,k}}\leq V(s,0)+\varepsilon.
\]
Set%
\[
\bar{u}^{k}=\sum_{j_{1}=1}^{l^{k}}\sum_{j_{2}=1}^{l^{k}}I_{\{ \tilde{X}%
_{s}^{t,x,u}\in D_{j_{1}}^{k}\}}I_{\{B_{s}^{1}-B_{t}^{1}\in E_{j_{2}}^{k}%
\}}\bar{u}^{j_{1}j_{2},k}+I_{\{ \tilde{X}_{s}^{t,x,u}\in\lbrack-ke,ke)^{c}\}
\cup\{B_{s}^{1}-B_{t}^{1}\in\lbrack-k,k)^{c}\}}\bar{u}^{0,k},
\]
by the similar analysis as in the proof of Lemma \ref{le-dpp-3}, we can get%
\[
V(s,\xi^{k,u})\leq Y_{s}^{s,\xi^{k,u},\bar{u}^{k}}\leq V(s,\xi^{k,u}%
)+\varepsilon.
\]
By Theorem \ref{BSDE-est}, there exists a constant $C_{3}>0$ depending on $T$,
$G$ and $C$ such that%
\[
\mathbb{\hat{E}}[|Y_{s}^{s,X_{s}^{t,x,u},\bar{u}^{k}}-Y_{s}^{s,\xi^{k,u}%
,\bar{u}^{k}}|^{2}]\leq C_{3}\mathbb{\hat{E}}[|X_{s}^{t,x,u}-\xi^{k,u}|^{2}].
\]
Set $\tilde{u}(r)=u(r)I_{[t,s]}(r)+\bar{u}^{k}(r)I_{(s,T]}(r)$, it is easy to
check that $\tilde{u}\in\mathbb{U}^{t}[t,T]$ and%
\[
V(t,x)\leq Y_{t}^{t,x,\tilde{u}}=\mathbb{G}_{t,s}^{t,x,u}[Y_{s}^{s,X_{s}%
^{t,x,u},\bar{u}^{k}}].
\]
By Theorem \ref{BSDE-est}, we get%
\begin{align*}
&  |\mathbb{G}_{t,s}^{t,x,u}[Y_{s}^{s,X_{s}^{t,x,u},\bar{u}^{k}}%
]-\mathbb{G}_{t,s}^{t,x,u}[V(s,X_{s}^{t,x,u})]|^{2}\\
&  \leq C_{3}\mathbb{\hat{E}}[|Y_{s}^{s,X_{s}^{t,x,u},\bar{u}^{k}}%
-V(s,X_{s}^{t,x,u})|^{2}]\\
&  \leq2C_{3}(\mathbb{\hat{E}}[|Y_{s}^{s,X_{s}^{t,x,u},\bar{u}^{k}}%
-Y_{s}^{s,\xi^{k,u},\bar{u}^{k}}|^{2}]+\mathbb{\hat{E}}[|Y_{s}^{s,\xi
^{k,u},\bar{u}^{k}}-V(s,X_{s}^{t,x,u})|^{2}])\\
&  \leq2C_{3}(C_{3}\mathbb{\hat{E}}[|X_{s}^{t,x,u}-\xi^{k,u}|^{2}%
]+\mathbb{\hat{E}}[(|V(s,\xi^{k,u})-V(s,X_{s}^{t,x,u})|+\varepsilon)^{2}])\\
&  \leq\frac{C_{4}(1+|x|^{4})}{k^{2}}+4C_{3}\varepsilon^{2}.
\end{align*}
Thus%
\[
V(t,x)-\frac{\sqrt{C_{4}(1+|x|^{4})}}{k}-(2\sqrt{C_{3}}+1)\varepsilon
\leq\underset{v(\cdot)\in\mathbb{U}^{t}[t,s]}{\inf}\mathbb{G}_{t,s}%
^{t,x,v}[V(s,X_{s}^{t,x,v})].
\]
Letting $k\rightarrow\infty$ first and then $\varepsilon\downarrow0$, we
obtain $V(t,x)\leq\underset{v(\cdot)\in\mathbb{U}^{t}[t,s]}{\inf}%
\mathbb{G}_{t,s}^{t,x,v}[V(s,X_{s}^{t,x,v})]$. The proof is complete.
\end{proof}

\begin{remark}
In the above proof, $\xi^{k,u}$ is called an \textquotedblleft implied
partition\textquotedblright\ of $X_{s}^{t,x,u}$.
\end{remark}

\begin{lemma}
\label{le-new-dpp2}Let Assumptions (A1) and (A2) hold. Then for any $t<s\leq
T$, $x\in\mathbb{R}^{n}$, we have%
\[
V(t,x)\leq\underset{u(\cdot)\in\mathbb{U}^{t}[t,s]}{\inf}\mathbb{G}%
_{t,s}^{t,x,u}[V(s,X_{s}^{t,x,u})].
\]

\end{lemma}

\begin{proof}
For each fixed $N>0$, we set $b^{i_{1},N}=(b^{i_{1}}\wedge N)\vee(-N)$,
$h_{ij}^{i_{1},N}=(h_{ij}^{i_{1}}\wedge N)\vee(-N)$, $\sigma_{i_{1}i_{2}}%
^{N}=(\sigma_{i_{1}i_{2}}\wedge N)\vee(-N)$ for $i_{1}\leq n$, $i_{2}\leq d$
and $b^{N}=(b^{1,N},\ldots,b^{n,N})^{T}$, $h_{ij}^{N}=(h_{ij}^{1,N}%
,\ldots,h_{ij}^{n,N})^{T}$, $\sigma^{N}=(\sigma_{i_{1}i_{2}}^{N})_{i_{1}i_{2}%
}$. Consider the following FBSDEs:%
\begin{align*}
dX_{s}^{t,x,u,N}  &  =b^{N}(s,X_{s}^{t,x,u,N},u_{s})ds+h_{ij}^{N}%
(s,X_{s}^{t,x,u,N},u_{s})d\langle B^{i},B^{j}\rangle_{s}+\sigma^{N}%
(s,X_{s}^{t,x,u,N},u_{s})dB_{s},\\
X_{t}^{t,x,u,N}  &  =x,
\end{align*}
and%
\begin{align*}
-dY_{s}^{t,x,u,N}  &  =f(s,X_{s}^{t,x,u,N},Y_{s}^{t,x,u,N},Z_{s}%
^{t,x,u,N},u_{s})ds\\
&  +g_{ij}(s,X_{s}^{t,x,u,N},Y_{s}^{t,x,u,N},Z_{s}^{t,x,u,N},u_{s})d\langle
B^{i},B^{j}\rangle_{s}-Z_{s}^{t,x,u,N}dB_{s}-dK_{s}^{t,x,u,N},\\
Y_{T}^{t,x,u,N}  &  =\Phi(X_{T}^{t,x,u,N}),\text{ \ \ \ }s\in\lbrack
t,T]\text{.}%
\end{align*}
We define%
\[
V^{N}(t,x)=\underset{u(\cdot)\in\mathbb{U}^{t}[t,T]}{\inf}Y_{t}^{t,x,u,N}.
\]
By Lemma \ref{le-new-dpp1}, we get for any $t<s\leq T$, $x\in\mathbb{R}^{n}$,%
\begin{equation}
V^{N}(t,x)\leq\underset{u(\cdot)\in\mathbb{U}^{t}[t,s]}{\inf}\mathbb{G}%
_{t,s}^{t,x,u,N}[V^{N}(s,X_{s}^{t,x,u,N})], \label{eq-dpp-2}%
\end{equation}
where $\mathbb{G}_{t,s}^{t,x,u,N}[\cdot]$ is defined as in $\mathbb{G}%
_{t,s}^{t,x,u}[\cdot]$. By Theorem \ref{SDE-est}, there exists a constant
$C_{1}>0$ depending on $T$, $n$, $U$, $G$ and $C$ such that for any
$u(\cdot)\in\mathbb{U}^{t}[t,T]$, $r\in\lbrack t,T]$,
\[
\mathbb{\hat{E}}[|X_{r}^{t,x,u}|^{4}]\leq C_{1}(1+|x|^{4}),
\]%
\begin{align*}
\mathbb{\hat{E}}[|X_{r}^{t,x,u,N}-X_{r}^{t,x,u}|^{2}]  &  \leq C_{1}%
\mathbb{\hat{E}}[\int_{t}^{T}(|b-b^{N}|^{2}+|h_{ij}-h_{ij}^{N}|^{2}%
+|\sigma-\sigma^{N}|^{2})(s,X_{s}^{t,x,u},u)ds]\\
&  \leq C_{1}\mathbb{\hat{E}}[\int_{t}^{T}\frac{1}{N^{2}}(|b|^{4}+|h_{ij}%
|^{4}+|\sigma|^{4})(s,X_{s}^{t,x,u},u)ds]\\
&  \leq\frac{C_{2}(1+|x|^{4})}{N^{2}},
\end{align*}
where $C_{2}$ depending on $T$, $n$, $U$, $G$ and $C$. By Theorem
\ref{BSDE-est}, there exists a constant $C_{3}>0$ depending on $T$, $G$ and
$C$ such that for any $u(\cdot)\in\mathbb{U}^{t}[t,T]$,%
\begin{align*}
|Y_{t}^{t,x,u,N}-Y_{t}^{t,x,u}|^{2}  &  \leq C_{3}\mathbb{\hat{E}}[|\Phi
(X_{T}^{t,x,u,N})-\Phi(X_{T}^{t,x,u})|^{2}+\int_{t}^{T}|X_{r}^{t,x,u,N}%
-X_{r}^{t,x,u}|^{2}dr]\\
&  \leq\frac{C_{4}(1+|x|^{4})}{N^{2}},
\end{align*}
where $C_{4}$ depending on $T$, $n$, $U$, $G$ and $C$. Thus we get%
\[
|V^{N}(t,x)-V(t,x)|\leq\underset{u(\cdot)\in\mathbb{U}^{t}[t,T]}{\sup}%
|Y_{t}^{t,x,u,N}-Y_{t}^{t,x,u}|\leq\frac{\sqrt{C_{4}(1+|x|^{4})}}{N}.
\]
It is easy to verify that Lemma \ref{le-dpp-4} still holds for $V^{N}$. Then
we can get%
\begin{align*}
&  |V^{N}(s,X_{s}^{t,x,u,N})-V(s,X_{s}^{t,x,u})|^{2}\\
&  \leq2(|V^{N}(s,X_{s}^{t,x,u,N})-V^{N}(s,X_{s}^{t,x,u})|^{2}+|V^{N}%
(s,X_{s}^{t,x,u})-V(s,X_{s}^{t,x,u})|^{2})\\
&  \leq2(L_{2}^{2}|X_{s}^{t,x,u,N}-X_{s}^{t,x,u}|^{2}+\frac{C_{4}%
(1+|X_{s}^{t,x,u}|^{4})}{N^{2}}).
\end{align*}
Similar to the proof of Lemma \ref{le-new-dpp1}, we can obtain for any
$u(\cdot)\in\mathbb{U}^{t}[t,s]$,
\begin{align*}
&  |\mathbb{G}_{t,s}^{t,x,u,N}[V^{N}(s,X_{s}^{t,x,u,N})]-\mathbb{G}%
_{t,s}^{t,x,u}[V(s,X_{s}^{t,x,u})]|^{2}\\
&  \leq C_{3}\mathbb{\hat{E}}[|V^{N}(s,X_{s}^{t,x,u,N})-V(s,X_{s}%
^{t,x,u})|^{2}+\int_{t}^{s}|X_{r}^{t,x,u,N}-X_{r}^{t,x,u}|^{2}dr].
\end{align*}
Thus%
\[
|\underset{u(\cdot)\in\mathbb{U}^{t}[t,s]}{\inf}\mathbb{G}_{t,s}%
^{t,x,u,N}[V^{N}(s,X_{s}^{t,x,u,N})]-\underset{u(\cdot)\in\mathbb{U}%
^{t}[t,s]}{\inf}\mathbb{G}_{t,s}^{t,x,u}[V(s,X_{s}^{t,x,u})]|\leq\frac
{\sqrt{C_{5}(1+|x|^{4})}}{N},
\]
where $C_{5}$ depending on $T$, $n$, $U$, $G$ and $C$. Taking $N\rightarrow
\infty$ in inequality (\ref{eq-dpp-2}), we obtain the result. The proof is complete.
\end{proof}

Now we give the proof of Theorem \ref{the-dpp-2}:

\begin{proof}
(1) By Lemma \ref{le-dpp-4}, we have for any $u(\cdot)\in\mathcal{U}[t,T]$,
\[
Y_{s}^{s,X_{s}^{t,x,u},u}\geq V(s,X_{s}^{t,x,u})\; \text{q.s.,}%
\]
where $Y_{s}^{s,X_{s}^{t,x,u},u}=Y_{s}^{t,x,u}$ is the solution of equation
(\ref{state equ-2}) at time $s$. Then, by the comparison theorem of $G$-BSDE,
we obtain%
\[
Y_{t}^{t,x,u}\geq\mathbb{G}_{t,s}^{t,x,u}[V(s,X_{s}^{t,x,u})]\; \text{q.s.},
\]
which leads to%
\[
V(t,x)\geq\underset{u(\cdot)\in\mathcal{U}[t,s]}{\text{ess}\inf}%
\mathbb{G}_{t,s}^{t,x,u}[V(s,X_{s}^{t,x,u})].
\]

(2) Now we prove the converse inequality.

By Theorem \ref{the-dpp-1}, we get%
\begin{align*}
\underset{u(\cdot)\in\mathcal{U}[t,s]}{\text{ess}\inf}\mathbb{G}_{t,s}%
^{t,x,u}[V(s,X_{s}^{t,x,u})]  &  =\underset{u(\cdot)\in\mathcal{U}%
^{t}[t,s]}{\inf}\mathbb{G}_{t,s}^{t,x,u}[V(s,X_{s}^{t,x,u})]\\
&  =\underset{u(\cdot)\in\mathbb{U}^{t}[t,s]}{\inf}\mathbb{G}_{t,s}%
^{t,x,u}[V(s,X_{s}^{t,x,u})].
\end{align*}
By Lemma \ref{le-new-dpp2}, we obtain%
\[
V(t,x)\leq\underset{u(\cdot)\in\mathbb{U}^{t}[t,s]}{\inf}\mathbb{G}%
_{t,s}^{t,x,u}[V(s,X_{s}^{t,x,u})].
\]
This completes the proof.
\end{proof}

The following lemma shows the continuity of $V$ in $t.$

\begin{lemma}
\label{Lem-dpp ine 2}The value function $V$ is $\frac{1}{2}$ H\"{o}lder
continuous in $t$.
\end{lemma}

\textbf{Proof. }Set $(t,x)\in\mathbb{R}^{n}\times\lbrack0,T]$ and $\delta>0.$
By dynamic programming principle, we have%
\[
V(t,x)=\underset{u(\cdot)\in\mathcal{U}^{t}[t,t+\delta]}{\inf}\mathbb{G}%
_{t,t+\delta}^{t,x,u}[V(t+\delta,X_{t+\delta}^{t,x,u})].
\]
Set $\bar{f}=\bar{g}_{ij}=0$, it is easy to verify that $(V(t+\delta
,x),0,0)_{s\in\lbrack t,t+\delta]}$ is the solution of $G$-BSDE
(\ref{state equ-2}) with terminal condition $\bar{Y}_{t+\delta}=V(t+\delta
,x)$. Thus by Proposition 5.1 in \cite{HJPS1}, there exists a constant
$C_{1}>0$ depending on $T$, $G$ and $C$ such that for any $u(\cdot
)\in\mathcal{U}^{t}[t,t+\delta]$,%
\begin{align*}
&  |\mathbb{G}_{t,t+\delta}^{t,x,u}[V(t+\delta,X_{t+\delta}^{t,x,u}%
)]-V(t+\delta,x)|^{2}\\
&  \leq C_{1}\mathbb{\hat{E}}[|V(t+\delta,X_{t+\delta}^{t,x,u})-V(t+\delta
,x)|^{2}+(\int_{t}^{t+\delta}(1+|X_{s}^{t,x,u}|+|V(t+\delta,x)|)ds)^{2}].
\end{align*}
By Lemmas \ref{le-dpp-4} and \ref{le-neww-dpp-5}, we can get%
\begin{align*}
&  |\mathbb{G}_{t,t+\delta}^{t,x,u}[V(t+\delta,X_{t+\delta}^{t,x,u}%
)]-V(t+\delta,x)|^{2}\\
&  \leq C_{2}\mathbb{\hat{E}}[(1+|x|^{2})\delta^{2}+\delta\int_{t}^{t+\delta
}|X_{s}^{t,x,u}|^{2}ds+|X_{t+\delta}^{t,x,u}-x|^{2}],
\end{align*}
where $C_{2}$ depending on $T$, $G$ and $C$. By Theorem \ref{SDE-est}, there
exists a constant $C_{3}>0$ depending on $T$, $n$, $U$, $G$ and $C$ such that
for any $u(\cdot)\in\mathcal{U}^{t}[t,t+\delta]$,%
\[
\mathbb{\hat{E}}[|X_{s}^{t,x,u}|^{2}]\leq C_{3}(1+|x|^{2}),\text{
}\mathbb{\hat{E}}[|X_{t+\delta}^{t,x,u}-x|^{2}]\leq C_{3}(1+|x|^{2})\delta.
\]
Then we obtain%
\[
|\mathbb{G}_{t,t+\delta}^{t,x,u}[V(t+\delta,X_{t+\delta}^{t,x,u}%
)]-V(t+\delta,x)|^{2}\leq C_{4}(1+|x|^{2})\delta,
\]
where $C_{4}$ depending on $T$, $n$, $U$, $G$ and $C$. Thus%
\begin{align*}
|V(t,x)-V(t+\delta,x)|  &  \leq\underset{u(\cdot)\in\mathcal{U}^{t}%
[t,t+\delta]}{\sup}|\mathbb{G}_{t,t+\delta}^{t,x,u}[V(t+\delta,X_{t+\delta
}^{t,x,u})]-V(t+\delta,x)|\\
&  \leq\sqrt{C_{4}(1+|x|^{2})}\delta^{\frac{1}{2}}.
\end{align*}
The proof is complete.

\section{The viscosity solution of HJB equation}

The following theorem gives the relationship between the value function $V$
and the second-order partial differential equation (\ref{hjb}).

\begin{theorem}
\label{viscosity} Let Assumptions (A1) and (A2) hold. $V$ is the value
function defined by (\ref{valuefunction0}). Then $V$ is the unique viscosity
solution of the following second-order partial differential equation:
\begin{align}
&  \partial_{t}V(t,x)+\underset{u\in U}{\inf}H(t,x,V,\partial_{x}%
V,\partial_{xx}^{2}V,u)=0,\label{hjb}\\
&  V(T,x)=\Phi(x),\quad\ \ x\in\mathbb{R}^{n},\nonumber
\end{align}
where%
\[%
\begin{array}
[c]{cl}%
H(t,x,v,p,A,u)= & G(F(t,x,v,p,A,u))+\langle p,b(t,x,u)\rangle+f(t,x,v,\sigma
^{T}(t,x,u)p,u),\\
F_{ij}(t,x,v,p,A,u)= & (\sigma^{T}(t,x,u)A\sigma(t,x,u))_{ij}+2\langle
p,h_{ij}(t,x,u)\rangle\\
& +2g_{ij}(t,x,v,\sigma^{T}(t,x,u)p,u),
\end{array}
\]
$(t,x,v,p,A,u)\in\lbrack0,T]\times\mathbb{R}^{n}\times\mathbb{R}%
\times\mathbb{R}^{n}\times\mathbb{S}_{n}\times U$, $G$ is defined by equation
(\ref{eq-G}).
\end{theorem}

For simplicity, we only consider the case $h_{ij}=g_{ij}=0$.

Suppose $\varphi\in C_{b,Lip}^{2,3}([0,T]\times\mathbb{R}^{n})$. Define%
\begin{equation}%
\begin{array}
[c]{rl}%
F_{1}(s,x,y,z,u)= & \partial_{s}\varphi(s,x)+f(s,x,y+\varphi(s,x),z+\sigma
^{T}(s,x,u)\partial_{x}\varphi(s,x),u)\\
& +\langle b(s,x,u),\partial_{x}\varphi(s,x)\rangle,\\
F_{2}^{ij}(s,x,u)= & \frac{1}{2}\langle\partial_{xx}^{2}\varphi(s,x)\sigma
_{i}(s,x,u),\sigma_{j}(s,x,u)\rangle.
\end{array}
\label{EQS-1}%
\end{equation}

Consider the following G-BSDEs: $\forall s\in\lbrack t,t+\delta],$%
\begin{equation}%
\begin{array}
[c]{rl}%
Y_{s}^{1,u}= & \int_{s}^{t+\delta}F_{1}(r,X_{r}^{t,x,u},Y_{r}^{1,u}%
,Z_{r}^{1,u},u_{r})dr+\int_{s}^{t+\delta}F_{2}^{ij}(r,X_{r}^{t,x,u}%
,u_{r})d\langle B^{i},B^{j}\rangle_{r}\\
& -\int_{s}^{t+\delta}Z_{r}^{1,u}dB_{r}-(K_{t+\delta}^{1,u}-K_{s}^{1,u}),
\end{array}
\label{auxilary-1}%
\end{equation}
and%
\begin{equation}
Y_{s}^{u}=\varphi(t+\delta,X_{t+\delta}^{t,x,u})+\int_{s}^{t+\delta}%
f(r,X_{r}^{t,x,u},Y_{r}^{u},Z_{r}^{u},u_{r})dr-\int_{s}^{t+\delta}Z_{r}%
^{u}dB_{r}-(K_{t+\delta}^{u}-K_{s}^{u}). \label{equa-1}%
\end{equation}

\begin{lemma}
\label{le-hjb-1}For each $s\in\lbrack t,t+\delta],$ we have%

\begin{equation}
Y_{s}^{1,u}=Y_{s}^{u}-\varphi(s,X_{s}^{t,x,u}). \label{auxilary-1-sol}%
\end{equation}

\end{lemma}

\begin{proof}
Applying It\^{o}'s formula to $\varphi(s,X_{s}^{t,x,u})$, we have%
\[
d(Y_{s}^{u}-\varphi(s,X_{s}^{t,x,u}))=dY_{s}^{1,u}.
\]
Since $Y_{t+\delta}^{u}-\varphi(t+\delta,X_{t+\delta}^{t,x,u})=Y_{t+\delta
}^{1,u}=0,$ we obtain%
\[
Y_{s}^{1,u}=Y_{s}^{u}-\varphi(s,X_{s}^{t,x,u}),\text{ \ \ }\forall s\in\lbrack
t,t+\delta].
\]
The proof is completed.
\end{proof}

Consider the G-BSDE: $\forall s\in\lbrack t,t+\delta],$%

\begin{equation}%
\begin{array}
[c]{rl}%
Y_{s}^{2,u}= & \int_{s}^{t+\delta}F_{1}(r,x,Y_{r}^{2,u},Z_{r}^{2,u}%
,u_{r})dr+\int_{s}^{t+\delta}F_{2}^{ij}(r,x,u_{r})d\langle B^{i},B^{j}%
\rangle_{r}\\
& -\int_{s}^{t+\delta}Z_{r}^{2,u}dB_{r}-(K_{t+\delta}^{2,u}-K_{s}^{2,u}).
\end{array}
\label{auxilary-2}%
\end{equation}

We have the following estimates.

\begin{lemma}
\label{le-hjb-2}For each $u(\cdot)\in\mathcal{U}^{t}[t,t+\delta]$, we have
\begin{equation}
\mid Y_{t}^{1,u}-Y_{t}^{2,u}\mid\leq L_{4}\delta^{3/2}, \label{auxilary-ine}%
\end{equation}
where $L_{4}$ is a positive constant dependent on $x$ and independent of
$u(\cdot)$.
\end{lemma}

\begin{proof}
By Proposition 5.1 in \cite{HJPS1}, there exists a constant $C_{1}>0$
depending on $T$, $G$ and $C$ such that for any $u(\cdot)\in\mathcal{U}%
^{t}[t,t+\delta]$,%
\[
|Y_{t}^{1,u}-Y_{t}^{2,u}|^{2}\leq C_{1}\mathbb{\hat{E}}[(\int_{t}^{t+\delta
}\hat{F}_{r}dr)^{2}],
\]
where
\begin{align*}
\hat{F}_{r}  &  =|F_{1}(r,X_{r}^{t,x,u},Y_{r}^{2,u},Z_{r}^{2,u},u_{r}%
)-F_{1}(r,x,Y_{r}^{2,u},Z_{r}^{2,u},u_{r})|\\
&  +\sum_{i,j=1}^{d}|F_{2}^{ij}(r,X_{r}^{t,x,u},u_{r})-F_{2}^{ij}(r,x,u_{r})|.
\end{align*}
Note that $\varphi\in C_{b,Lip}^{2,3}([0,T]\times\mathbb{R}^{n})$, it is easy
to verify that
\[
\hat{F}_{r}\leq C_{2}(|X_{r}^{t,x,u}-x|+|X_{r}^{t,x,u}-x|^{2}),
\]
where $C_{2}$ is dependent on $x$ and\ independent of $u(\cdot)$. By Theorem
\ref{SDE-est}, we can obtain that for any $p\geq2$,
\[
\mathbb{\hat{E}}[\underset{r\in\lbrack t,t+\delta]}{\sup}\mid X_{r}%
^{t,x,u}-x\mid^{p}]\leq C_{3}(1+|x|^{p})\delta^{p/2},
\]
where $C_{3}$ is independent of $u(\cdot)$. Then by H\"{o}lder's inequality we
can deduce that $|Y_{t}^{1,u}-Y_{t}^{2,u}|\leq L_{4}\delta^{3/2}$, where
$L_{4}$ is dependent on $x$ and\ independent of $u(\cdot)$. This completes the proof.
\end{proof}

Now we compute $\underset{u(\cdot)\in\mathcal{U}^{t}[t,t+\delta]}{\inf}%
Y_{t}^{2,u}.$

\begin{lemma}
\label{Lem-auxilary-3}We have
\[
\underset{u(\cdot)\in\mathcal{U}^{t}[t,t+\delta]}{\inf}\text{ }Y_{t}%
^{2,u}=Y_{t}^{0},
\]
where $Y^{0}$ is the solution of the following ordinary differential equation%
\begin{equation}
-dY_{s}^{0}=F_{0}(s,x,Y_{s}^{0},0)ds,\text{ \ \ }Y_{t+\delta}^{0}=0,\text{
}s\in\lbrack t,t+\delta] \label{auxilary-3}%
\end{equation}
and%
\[%
\begin{array}
[c]{rl}%
F_{0}(s,x,y,z):= & \underset{u\in U}{\inf}[F_{1}(s,x,y,z,u)+2G(F_{2}(s,x,u))].
\end{array}
\]

\end{lemma}

In order to prove Lemma \ref{Lem-auxilary-3}, we need the following property
of the decreasing $G$-martingale.

\begin{lemma}
\label{le-hjb-4}Suppose that $(M_{s})_{t\leq s\leq t+\delta}$ is a decreasing
$G$-martingale. Then there exists a $Q\in\mathcal{P}$ such that%
\[
M_{t+\delta}=M_{t},\;Q-\text{a.s.}.
\]

\end{lemma}

\begin{proof}
By the representation of $G$-expectation, we know that%
\[
\mathbb{\hat{E}}[M_{t+\delta}-M_{t}]=\underset{P\in\mathcal{P}}{\sup}%
E_{P}[M_{t+\delta}-M_{t}].\;
\]
Thus there exist $Q_{k}\in\mathcal{P}$, $k=1,2,...,$ such that%
\[
E_{Q_{k}}[M_{t+\delta}-M_{t}]\uparrow\mathbb{\hat{E}}[M_{t+\delta}%
-M_{t}]=\mathbb{\hat{E}}[\mathbb{\hat{E}}_{t}[M_{t+\delta}]-M_{t}]=0.
\]
Since $\mathcal{P}$ is weakly compact, there exist $Q\in\mathcal{P}$ and a
subsequence $(Q_{k_{i}})$ of $(Q_{k})$ such that $Q_{k_{i}}$ converges weakly
to $Q$. By Lemma 29 in \cite{DHP11}, then we get%
\[
E_{Q}[M_{t+\delta}-M_{t}]=\lim_{i\rightarrow\infty}E_{Q_{k_{i}}}[M_{t+\delta
}-M_{t}]=0.
\]
Note that $M_{t+\delta}-M_{t}\leq0$, q.s.. Thus, we obtain that%
\[
M_{t+\delta}=M_{t},\;Q-\text{a.s.}.
\]
This completes the proof.
\end{proof}

Now we give the proof of Lemma \ref{Lem-auxilary-3}.

\begin{proof}
(1) We first prove that for any $u(\cdot)\in\mathcal{U}^{t}[t,t+\delta]$,%
\[
Y_{t}^{2,u}\geq Y_{t}^{0}.
\]

Note that%
\[%
\begin{array}
[c]{rl}%
Y_{s}^{2,u}= & \int_{s}^{t+\delta}[F_{1}(r,x,Y_{r}^{2,u},Z_{r}^{2,u}%
,u_{r})+2G(F_{2}(r,x,u_{r}))]dr\\
& +[\int_{s}^{t+\delta}F_{2}^{ij}(r,x,u_{r})d\langle B^{i},B^{j}\rangle
_{r}-\int_{s}^{t+\delta}2G(F_{2}(r,x,u_{r}))dr]\\
& -\int_{s}^{t+\delta}Z_{r}^{2,u}dB_{r}-(K_{t+\delta}^{2,u}-K_{s}%
^{2,u}),\text{ q.s.}.
\end{array}
\]
It is easy to verify that
\[
M_{s}=\int_{t}^{s}F_{2}^{ij}(r,x,u_{r})d\langle B^{i},B^{j}\rangle_{r}%
-\int_{t}^{s}2G(F_{2}(r,x,u_{r}))dr
\]
is a decreasing $G$-martingale.

By Lemma \ref{le-hjb-4}, there exists a $Q\in\mathcal{P}$ such that for
$s\in\lbrack t,t+\delta]$,%
\[
M_{s}=0,\;Q-\text{a.s.}.
\]
Then
\[%
\begin{array}
[c]{rl}%
Y_{s}^{2,u}= & \int_{s}^{t+\delta}[F_{1}(r,x,Y_{r}^{2,u},Z_{r}^{2,u}%
,u_{r})+2G(F_{2}(r,x,u_{r}))]dr\\
& -\int_{s}^{t+\delta}Z_{r}^{2,u}dB_{r}-(K_{t+\delta}^{2,u}-K_{s}%
^{2,u}),\;Q-\text{a.s.}.
\end{array}
\]

Consider the following BSDE: for $s\in\lbrack t,t+\delta],$
\[%
\begin{array}
[c]{cl}%
Y_{s}^{0}= & \int_{s}^{t+\delta}F_{0}(r,x,Y_{r}^{0},Z_{r}^{0})dr-\int%
_{s}^{t+\delta}Z_{r}^{0}dB_{r}-(K_{t+\delta}^{0}-K_{s}^{0}),\;Q-\text{a.s.}.
\end{array}
\]
Since $F_{0}$ is a deterministic function, we obtain that $Z_{s}^{0}%
=0,K_{s}^{0}=0$ and $Y_{s}^{0}$ is just the solution of equation
(\ref{auxilary-3}). Note that $(-K_{s}^{2,u})_{s\in\lbrack t,t+\delta]}$ is a
increasing process and $F_{1}(r,x,y,z,u_{r})+2G(F_{2}(r,x,u_{r}))\geq
F_{0}(r,x,y,z)$, then by the comparison theorem of classical BSDE (under the
reference probability measure $Q$), we deduce that
\[
Y_{t}^{2,u}\geq Y_{t}^{0}.
\]

(2) We denote the class of all deterministic controls in $\mathcal{U}%
^{t}[t,t+\delta]$ by $\mathcal{U}_{1}$. Then, for every $u(\cdot
)\in\mathcal{U}_{1},$ $Y^{2,u}$ is the solution of the following ordinary
differential equation:%
\[%
\begin{array}
[c]{l}%
-dY_{s}^{2,u}=[F_{1}(s,x,Y_{s}^{2,u},0,u_{s})+2G(F_{2}(s,x,u_{s}))]ds,\text{
\ \ }s\in\lbrack t,t+\delta],\\
Y_{t+\delta}^{2,u}=0.
\end{array}
\]
It is easy to check that%
\[
Y_{t}^{0}=\underset{u(\cdot)\in\mathcal{U}_{1}}{\inf}Y_{t}^{2,u}%
\geq\underset{u(\cdot)\in\mathcal{U}^{t}[t,t+\delta]}{\inf}\text{ }Y_{t}%
^{2,u}.
\]
This completes the proof.
\end{proof}

Finally we give the proof of Theorem \ref{viscosity}.

\begin{proof}
The uniqueness of viscosity solution of equation (\ref{hjb}) can be proved
similarly as in Theorem 6.1 in \cite{BL}, we only prove that $V$ is a
viscosity solution of equation (\ref{hjb}). By Lemmas \ref{le-dpp-4} and
\ref{Lem-dpp ine 2}, $V$ is a continuous functions on $[0,T]\times
\mathbb{R}^{n}$. We first prove that $V$ is the subsolution of (\ref{hjb}).

Given $t\leq T$ and $x\in\mathbb{R}^{n}$, suppose $\varphi\in C_{b,Lip}%
^{2,3}([0,T]\times\mathbb{R}^{n})$ such that $\varphi(t,x)=V(t,x)$ and
$\varphi\geq V$ on $[0,T]\times\mathbb{R}^{n}$. By Theorem\ \ref{the-dpp-2},
we have%
\[%
\begin{array}
[c]{rl}%
V(t,x)= & \underset{u(\cdot)\in\mathcal{U}^{t}[t,t+\delta]}{\inf}%
\mathbb{G}_{t,t+\delta}^{t,x,u}[V(t+\delta,X_{t+\delta}^{t,x,u})].
\end{array}
\]

Note that $\varphi\geq V$ on $[0,T]\times\mathbb{R}^{n}$. Then by comparison
theorem, we get
\[
\underset{u(\cdot)\in\mathcal{U}^{t}[t,t+\delta]}{\inf}\text{ }\{
\mathbb{G}_{t,t+\delta}^{t,x,u}[\varphi(t+\delta,X_{t+\delta}^{t,x,u}%
)]-\varphi(t,x)\} \geq0.
\]
By equality (\ref{auxilary-1-sol}), we have
\[
\underset{u(\cdot)\in\mathcal{U}^{t}[t,t+\delta]}{\inf}Y_{t}^{1,u}\geq0.
\]
By inequality (\ref{auxilary-ine}) and Lemma \ref{Lem-auxilary-3}, we get
\[
\underset{u(\cdot)\in\mathcal{U}^{t}[t,t+\delta]}{\inf}Y_{t}^{2,u}\geq
-L_{4}\delta^{3/2}%
\]
and%
\[
Y_{t}^{0}\geq-L_{4}\delta^{3/2}.
\]
Thus%
\[
-L_{4}\delta^{1/2}\leq\delta^{-1}Y_{t}^{0}=\delta^{-1}\int_{t}^{t+\delta}%
F_{0}(r,x,Y_{r}^{0},0)dr.
\]
Letting $\delta\rightarrow0$, we get $F_{0}(t,x,0,0)=\underset{u\in U}{\inf
}(F_{1}(t,x,0,0,u)+G(F_{2}(t,x,u)))\geq0$, which implies that $V$ is a
subsolution of (\ref{hjb}). Using the same method, we can\ prove $V$ is the
supersolution of (\ref{hjb}).\quad This completes the proof.$\ $
\end{proof}

\bigskip

\textbf{Acknowledgments}

We would like to thank S. Peng for many helpful discussions.

\end{document}